\documentclass[11pt]{article}
\usepackage[utf8]{inputenc}
\usepackage{amsmath}
\usepackage{amssymb}
\usepackage{amsthm}
\usepackage{xcolor}
\usepackage{wrapfig}
\usepackage{multirow}
\usepackage{graphicx}
\usepackage{grffile}
\usepackage[font=footnotesize,labelfont=bf]{caption}
\usepackage{anyfontsize}
\usepackage[font=footnotesize]{subcaption}
\usepackage{authblk}

\newcommand{\Poi}{\mathrm{Poi}}
\newcommand{\IG}{\mathrm{IG}}

\newcommand{\dd}{\mathrm{d}}

\newcommand{\EE}{\mathbb{E}}
\newcommand{\Var}{\mathrm{Var}}

\newcommand{\eps}{\varepsilon}

\newcommand{\xx}{\mathbf{x}}

\newtheorem{definition}{Definition}[section]
\newtheorem{theorem}[definition]{Theorem}

\newtheorem{remark}[definition]{Remark}
\newtheorem{lemma}[definition]{Lemma}

\newtheorem{assumption}[definition]{Assumption}

\title{The variation of the posterior variance and Bayesian sample size determination}

\author{ Jörg Martin and Clemens Elster }
\affil{Physikalisch-Technische Bundesanstalt (PTB)}

\begin{document}
\maketitle

\begin{abstract}
We consider Bayesian sample size determination using a criterion that utilizes the first two moments of the expected posterior variance. We study the resulting sample size in dependence on the chosen prior and explore the success rate for bounding the posterior variance below a prescribed limit under the true sampling distribution. Compared with sample size determination based on the expected average of the posterior variance the proposed criterion leads to an increase in sample size and significantly improved success rates. Generic asymptotic properties are proven, such as an asymptotic expression for the sample size and a sort of phase transition. Our study is illustrated using two real world datasets with Poisson and normally distributed data. Based on our results some recommendations are given. 
\end{abstract}

\section{Introduction}%
\label{sec:introduction}
Sample size determination (SSD) is the attempt to estimate the data size that is needed in order to meet a certain criterion \cite{Desu2012}. This task is usually performed at a planning stage before any data is actually measured or recorded so that especially in the context of high financial or temporal expenses a careful SSD becomes indispensable. In the design, say,  of animal experiments or clinical trials SSD can even have an ethical dimension \cite{Charan2013, Dell2002}. 
In this article we study a Bayesian method for SSD that limits the expected fluctuations of the uncertainty of the result. By ``uncertainty'' we will here mean (the square root of) the posterior variance. 
For $n$ data points $\xx_n = (x_1,\ldots,x_n)$ drawn from a sampling distribution $p(\xx_n|\theta)$ with parameter $\theta$ the posterior distribution is defined by
\begin{align}
\label{eq:posterior}
\pi(\theta|\xx_n) \propto \pi(\theta)\cdot p(\xx_n|\theta)\,,
\end{align}
where $\pi(\theta)$ denotes the prior for the parameter $\theta$. The posterior variance is then given as 
\begin{align}
	u_n^2 := \Var_{\theta \sim \pi(\theta|\xx_n)}(\theta) \,.
	\label{eq:PostVar}
\end{align}
In practice, a scientist performing an experiment might desire to specify her/his result with an according uncertainty, say 
\begin{align*}
	\hat{\theta}\pm u_n\,,
\end{align*}
with $u_n$ being the square root of $u_n^2$ as defined in \eqref{eq:PostVar} and with $\hat{\theta}$ being the posterior mean. In order for this result to be precise enough the scientist might desire to fulfill a condition such as $u_n< \eps$ or, equivalently,
\begin{align}
\label{eq:SSD_Goal}
u_n^2 < \eps^2
\end{align}
for some small, positive $\eps$ that is chosen a priori. As the posterior distribution is dependent on the data $\xx_n$, so is $u_n^2$. Choosing an appropriate sample size $n$ so that \eqref{eq:SSD_Goal} is guaranteed \emph{before} $\xx_n$ is known is only possible for a few restricted scenarios, for instance Bernoulli distributed samples \cite{Turkkan1992, Joseph2019}. A more generally applicable criterion is to require instead of \eqref{eq:SSD_Goal} 
\begin{align}
	\overline{u_n^2} = \EE_{\xx_n\sim m(\xx_n)} [u_n^2] < \eps^2\,, 
	\label{eq:APVC}
\end{align}
where $m(\xx_n)=\int p(\xx_n|\theta) \pi(\theta) \dd \theta$ denotes the \emph{prior predictive}.
This is known as the \emph{average posterior variance criterion} (APVC) in the literature \cite{Wang2002, Turkkan1992, Santis2007}. 
For many standard cases explicit expressions for $\overline{u_n^2}$ can be derived. The usage of the prior predictive $m(\xx_n)$ is quite natural as it describes what is known about the data $\xx_n$ given our prior knowledge. 
We will denote the smallest $n$ such that the APVC \eqref{eq:APVC} is satisfied throughout this article by $\widetilde{n}_\eps$. 
In the literature many alternative criteria can be found that replace $u^2_n$ by some other, data dependent random variable $T(\xx_n)$, compare for instance \cite{Adcock1997, Lan2008, Santis2006, Wang2002, Rubin1998}. While we will stick in this article to the choice $T(\xx_n)=u^2_n$ many of the ideas presented here can, in principle, be translated to such approaches. \footnote{Section \ref{sec:asymptotics} of this article, for instance, relies on Assumption \ref{ass:bernstein}. Provided similar assumptions hold for $T(\xx_n)$ then all the proofs given there carry through.}

The APVC has a rather obvious drawback: it only guarantees \eqref{eq:SSD_Goal} to hold on average. Consequently this might lead to an uncertainty $u_{\widetilde{n}_\eps}$ that is well beyond $\eps$ for certain data samples $\xx_{\widetilde{n}_\eps}$, compare for instance Figure \ref{subfig:Football_acc_APVC} below. To get a better grasp on the variability of the uncertainty one can make more extensive usage of the prior predictive $m(\xx_n)$ \cite{Santis2006, Santis2007, Brutti2008, Gubbiotti2011, Sambucini2008}. This article aims at studying the behavior of such criteria and thereby to give some guidance or, at least, a deeper understanding for a SSD based on $m(\xx_n)$.  
In order to do so we will use a extension of the APVC \eqref{eq:APVC} which we will call the \emph{variation of the posterior variance criterion} (VPVC) and which takes the form
\begin{align}
	\label{eq:VPVC}
	\overline{u_n^2} + k\cdot \Delta u_n^2 < \eps^2\,.
\end{align}
where $k$ is a parameter to be chosen  and where
\begin{align}
	\label{eq:Deltau2}
	\Delta u_n^2 =  (\Var_{x\sim m(\xx_n)} (u_n^2))^{1/2} \,.
\end{align}

\begin{wrapfigure}{l}{0.5\textwidth}
	\includegraphics[width=0.48\textwidth]{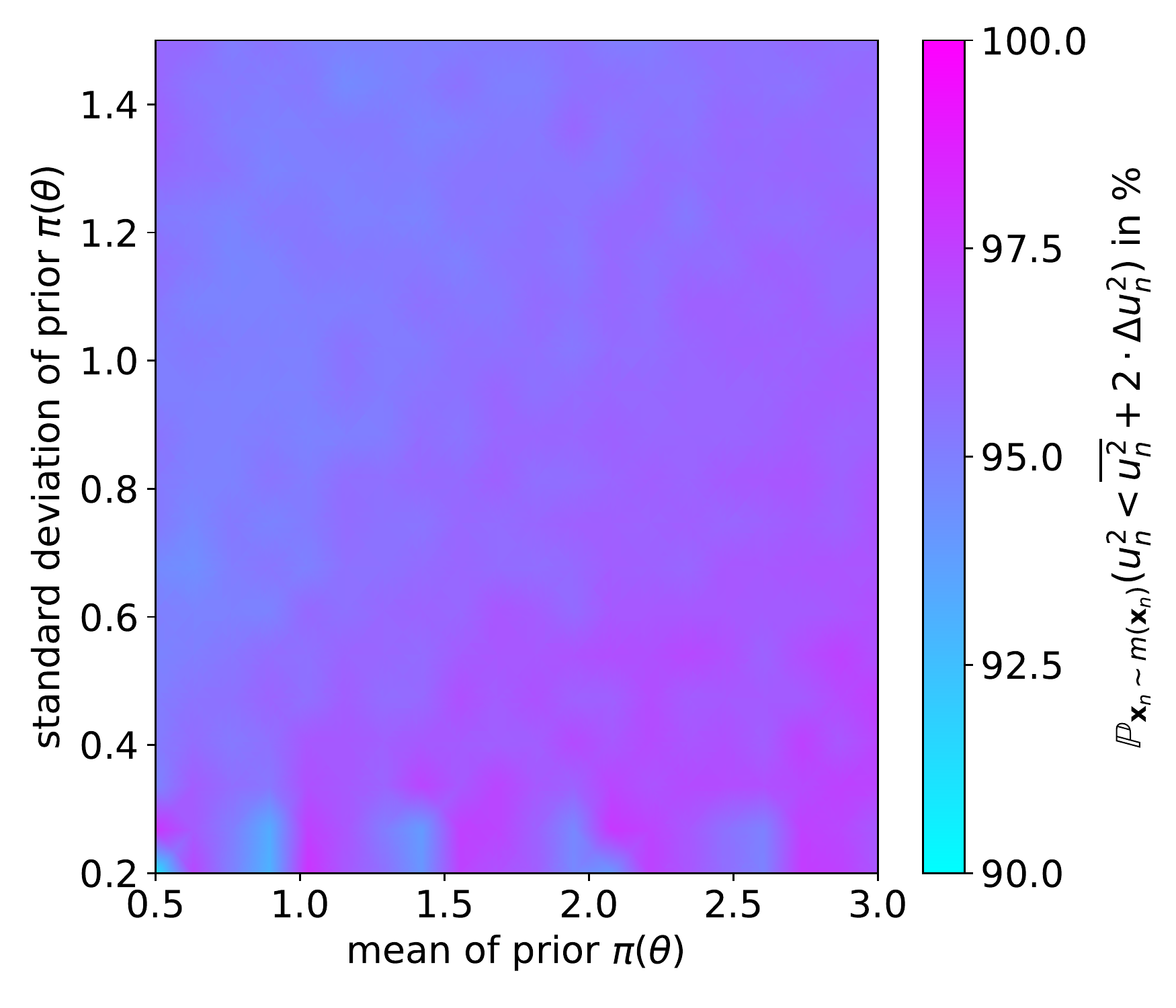}
	\caption{Probability for $u_n^2$ to be below $\overline{u_n^2}+2\cdot \Delta u_n^2$ when drawing $\xx_n \sim m(\xx_n)$. Setup is as in Section \ref{subsec:Poisson} below and sample sizes $n$ are as in Figure \ref{subfig:Football_n_VPVC}. Values range between 90.96-97.97\% with an average around 95.79\%.}
	\label{fig:ProbVPVC}
\end{wrapfigure}
Throughout this article we will denote the smallest $n$ such that \eqref{eq:VPVC} is satisfied by $n_\eps$. We will see in Section \ref{sec:SSD} below that taking \eqref{eq:VPVC} into account leads to substantially different sample sizes than the sole consideration of the APVC \eqref{eq:APVC} and to a better compliance with \eqref{eq:SSD_Goal}.  
We will often take $k=2$ in this article (loosely motivated from the normal distribution) but will show that for $\eps\rightarrow 0$ there is an optimal $k_\ast$. We will give some ideas on how to guess $k_\ast$ in practice.
Moreover, we will provide an asymptotic formula for the sample size $n_\eps$ in the small $\eps$ regime.  

The VPVC \eqref{eq:VPVC} is of course not as exhaustive in its description of the variability of $u_n^2$ as the full law of $u_n^2$ under $m(\xx_n)$. 
In Figure \ref{fig:ProbVPVC} we can see however that for $k=2$ and the setup from Section \ref{subsec:Poisson} the left hand side of \eqref{eq:VPVC} roughly covers 95.8\% of the law of $u_n^2$ under $m(\xx_n)$ and thus pretty accurately coincides with the intuition (for a normal random variable we would expect 97.7\%). 
On the other hand, this simplification is quite convenient for our purposes: it provides us with explicit expressions, spares us numerical issues and allows for a rather concise discussion of asymptotic properties in Section \ref{sec:asymptotics}. We will discuss our method for two common cases: for Poisson and normally distributed data and illustrate our discussion with real world datasets.

The authors are, to the best of their knowledge, not aware of work in the literature that considers a criterion in the exact same shape of \eqref{eq:VPVC}.
In \cite{Turkkan1992} Pham-Gia and Turkkan consider an object such as $\Delta u_n^2$ from \eqref{eq:Deltau2} for a Binomial distribution but apply it in a different manner. Our approach is inspired from the quite common idea of using $m(\xx_n)$ for studying $u_n^2$. We believe that the discussion in this article deepens the understanding of SSD methods build on $m(\xx_n)$ in general.

The article is organized as follows: Section \ref{sec:SSD} discusses the application of the VPVC to Poisson and normally distributed data. We will compare our results to the ones of the simpler AVPC method \eqref{eq:APVC}, visualize the effect of using a prior that is inconsistent with the underlying parameter and debate how a conservative SSD could be performed. For this purpose we will use actual datasets, namely the goals from international football matches in the years 2015-2019 \cite{Football} and the length of pop songs from the \emph{Million song dataset} \cite{MillionSongs}. In Section \ref{sec:asymptotics} we will look at the behavior of the VPVC for $\eps\rightarrow 0$. The results from that section, especially Theorem \ref{thm:conservative}, indicate that SSD methods based on $m(\xx_n)$ exhibit some sort of phase transition in limit. We will show that a $k_\ast$ exists such that for $k>k_\ast$ a SSD based on the VPVC \eqref{eq:VPVC} will ensure that \eqref{eq:SSD_Goal} will \emph{asymptotically be true with probability 1}. As $k_\ast$ will depend on the (usually unknown) true value of $\theta$, we will discuss a method how to get an upper bound based on the prior knowledge. Moreover, in Lemma \ref{lem:n_eps} we will provide a generic, asymptotic formula for the sample $n_\eps$ predicted the VPVC.

\section{SSD based on the variation of the posterior variance}
\label{sec:SSD}

In this section we study the effects of using the VPVC as proposed in \eqref{eq:VPVC} and discuss its dependency on the prior knowledge and the true parameter.
The discussion is first carried out for the case of a single parameter and Poisson distributed data. In Section \ref{subsec:normal} we then consider an example that involves nuisance parameters and normally distributed data.

\vspace{1.5cm}
\subsection{Single Parameter: Poisson distributed data}%
\label{subsec:Poisson}
\begin{wrapfigure}{l}{0.5\textwidth}
	\includegraphics[width=0.48\textwidth]{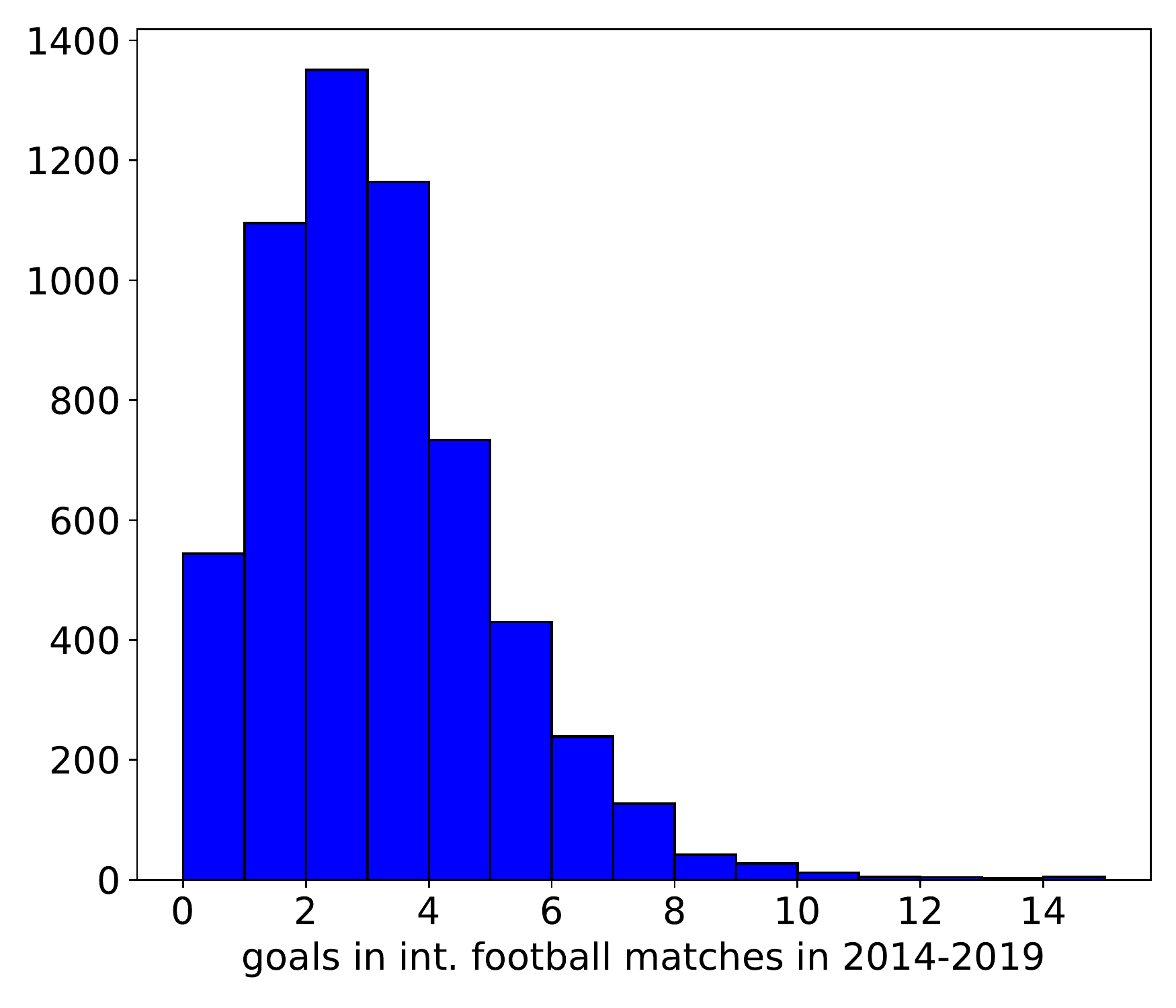}
	\caption{Scored goals in international football games between 2014 and 2019, taken from \cite{Football}.}
	\label{fig:goals}
\end{wrapfigure}%
The histogram in Figure \ref{fig:goals} shows the number of goals scored in 5784 international football matches that were played between 2014 and 2019, taken form \cite{Football}. 
Suppose we want to get an estimate of the \emph{average number of goals} $\theta$ that are scored in a game (from both teams). Of course, if we have the full dataset depicted in Figure \ref{fig:goals} we can simple take the sample mean, which yields in fact an average of $2.71$ goals per game in the mentioned time range. 

If we are \emph{not} in possession of the full dataset the question arises how many samples, that is football matches in this case, do we need before we can make a ``decent'' guess about $\theta$. To formalize this question let us suppose that the data for $n$ games follows the product of $n$ Poisson distributions with  the (unknown) parameter of interest $\theta$: 
\footnote{This is an approximation. The ratio of variance and mean is around 1.3 for the full dataset, however to keep the setup simple we will stick to the Poisson assumption and treat the mean, i.e. the MLE, of 2.71 as the true parameter.}
\begin{align}
\label{eq:PoissonSamplingDistribution}
p(\xx_n|\theta)=\prod_{i=1}^n \mathrm{Poi}(x_i|\theta) = \prod_{i=1}^n \frac{\theta^{x_i}}{x_i! } e^{-\theta} \,,
\end{align}
where $\xx_n=(x_1,\ldots,x_n)$  and each $x_i$ for $i=1,\ldots,n$ should be read as the total number of goals in game $i$. 
How large should we choose $n$ to get an estimate for $\theta$? A somewhat natural idea is to require that $n$ should be large enough so that the uncertainty $u_n$ we can specify for the result is smaller than an $\eps$, say $\eps=0.3$. Recall however that we want to find $n$ before we measure any data, that is, in our example, before we know any scores. Usually, $u_n$ will depend on the available data, so how can we find  $n$ before we know $\xx_n$?

\begin{figure}[t]
\centering
	\begin{subfigure}[t]{0.49\textwidth}
		\centering
		\includegraphics[width=\textwidth]{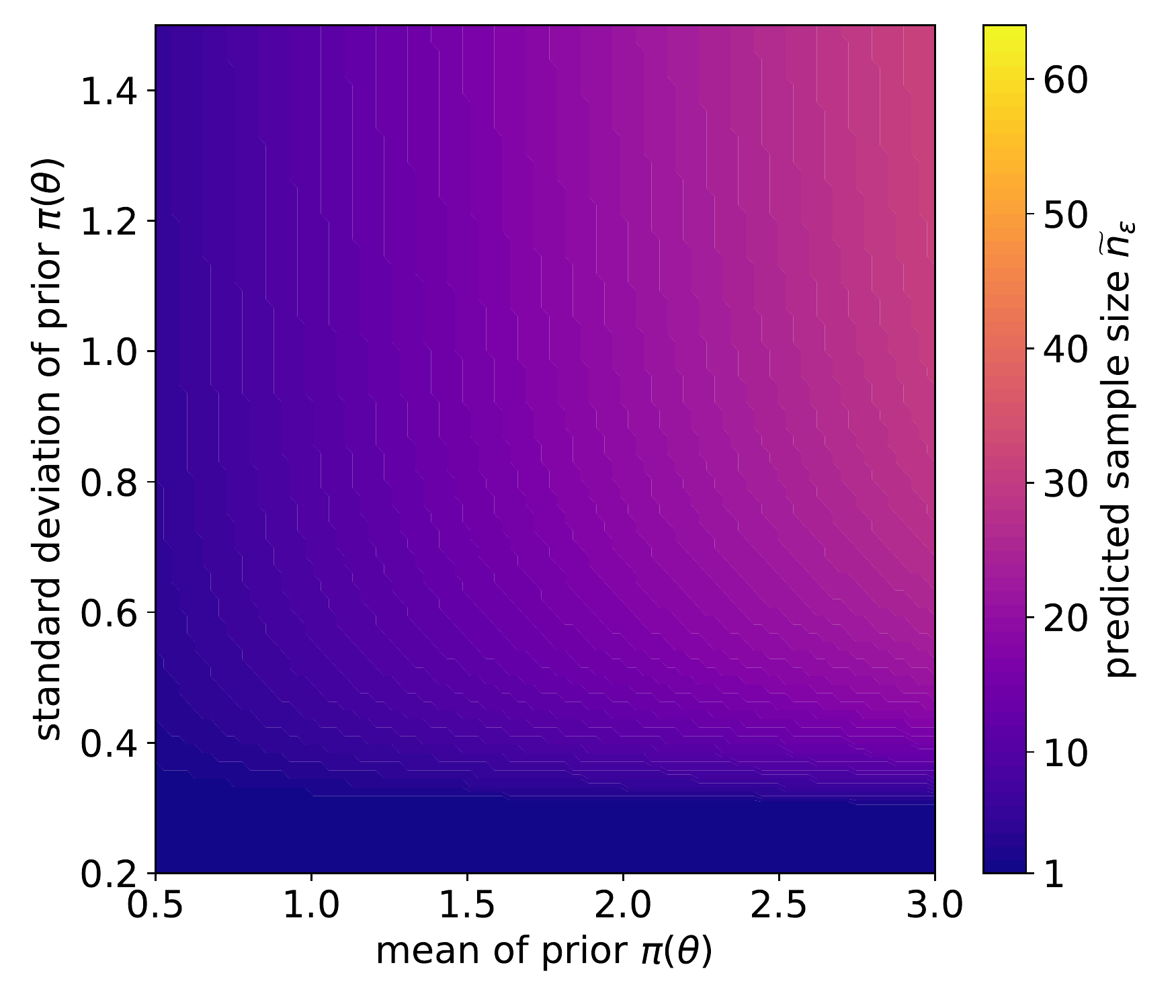}
		\caption{APVC: sample sizes $\widetilde{n}_\eps$}
		\label{subfig:Football_n_APVC}
	\end{subfigure}%
~
	\begin{subfigure}[t]{0.49\textwidth}
		\centering
		\includegraphics[width=\textwidth]{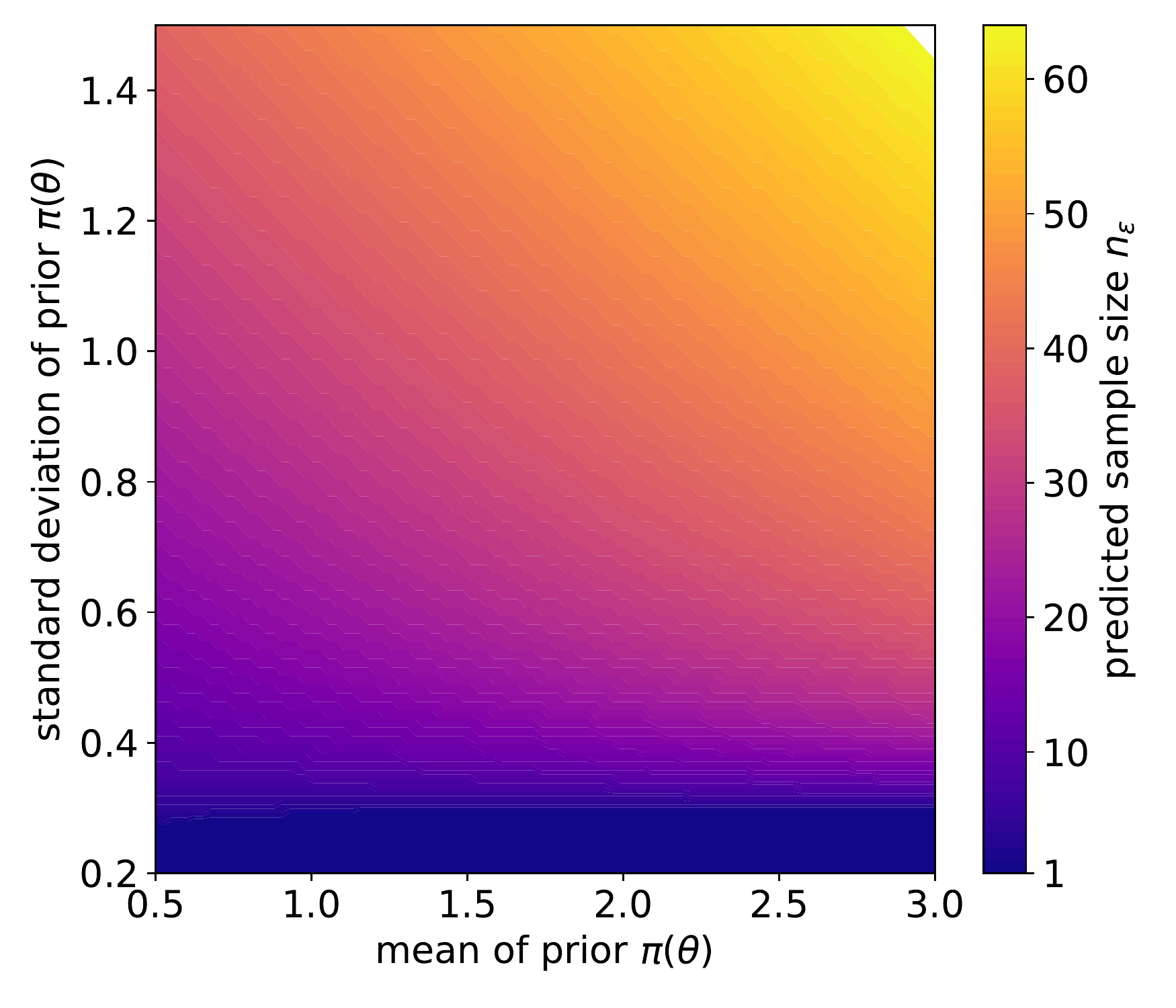}
		\caption{VPVC: sample sizes $n_\eps$}
		\label{subfig:Football_n_VPVC}
	\end{subfigure}
	\caption{\emph{left}: sample sizes predicted for $\Poi(\xx_n|\theta)$ distributed data and various priors $\pi(\theta)=\mathrm{Gamma}(\theta|\alpha,\beta)$ for $\eps=0.3$ by the APVC \eqref{eq:FootballAPVC}. \emph{right}: sample sizes $n_\eps$ predicted by the VPVC \eqref{eq:FootballVPVC} for the same setup.}
\end{figure}

At this point Bayesian statistics comes in handy as it allows us to use prior knowledge for this purpose, condensed in a prior distribution $\pi(\theta)$. We will here use a gamma distribution $\pi(\theta)=\mathrm{Gamma}(\theta|\alpha,\beta)$ with shape parameter $\alpha$ and rate $\beta$. We can then compute the prior predictive $m(\xx_n)$ as in \eqref{eq:APVC} which gives us a distribution for the data $\xx_n$ given our prior knowledge. As uncertainty $u_n$ for our result we take the square root of the posterior variance $u_n^2=\frac{\alpha+\sum_{i=1}^n x_i}{(n+\beta)^2}$. The inequality $u_n^2 < \eps^2$ has no solution for $n$ that holds for any $\xx_n$. However, we can require this inequality to hold on average over the available information $m(\xx_n)$ about $\xx_n$. In other words: choose the smallest $n\geq 1$ such that
\begin{align}
	\label{eq:FootballAPVC}
	\overline{u_n^2}=\EE_{\xx_n\sim m(\xx_n)}[u_n^2] = \frac{\alpha}{\beta} \cdot \frac{1}{n+\beta} < \eps^2 \,.
\end{align}
We referred to this as the APVC in the introduction and denoted the corresponding $n$ by $\widetilde{n}_\eps$. The result of such an SSD for $\eps=0.3$ is shown in Figure \ref{subfig:Football_n_APVC}: for various values of the prior mean $\EE_{\theta\sim\pi(\theta)}[\theta]= \frac{\alpha}{\beta} $ and its standard deviation $( (\Var_{\theta \sim \pi(\theta)}(\theta) ))^{1/2}= \frac{\alpha^{1/2}}{\beta} $ we plotted the sample size $\widetilde{n}_\eps$ predicted by \eqref{eq:FootballAPVC}. Naturally, as the mean of $\pi(\theta)$ increases the sample size $\widetilde{n}_\eps$ increases as well since $\theta$ is linked to the variance of the Poisson distributed data. The standard deviation of $\pi(\theta)$ seems to have only a minor influence on the sample size with one exception: below a certain threshold the prior variance pushes the posterior variance into the right ballpark so that even a minimal sample size of $\widetilde{n}_\eps=1$ is enough to fulfill \eqref{eq:FootballAPVC}. 

Figure \ref{subfig:Football_acc_APVC} illustrates for which prior choices the SSD based on the APVC is successful based on the football goal dataset we depicted in Figure \ref{fig:goals}: for each prior mean and standard deviation and the corresponding sample sizes $\widetilde{n}_\eps$ from Figure \ref{subfig:Football_n_APVC} 1000 random samples of size $\widetilde{n}_\eps$ were drawn from the football dataset and the corresponding $u_{\widetilde{n}_\eps}^2$ were computed. 
Figure \ref{subfig:Football_acc_APVC} shows the fraction of $u_{\widetilde{n}_\eps}^2$ that is actually below $\eps^2$.
For comparison the ``true'' value (the average goal number in the full dataset shown in Figure \ref{fig:goals}) is depicted by the dashed line.
The result is rather disappointing. Only beyond the true $\theta$ of $2.71$ the quota rises above $60\%$.  The reason for this unpleasant effect is quite apparent: by construction, the APVC only requires $u_{\widetilde{n}_\eps}^2$ to be small enough \emph{on average}.  
\begin{figure}[t]
\centering
	\begin{subfigure}[t]{0.49\textwidth}
		\centering
		\includegraphics[width=\textwidth]{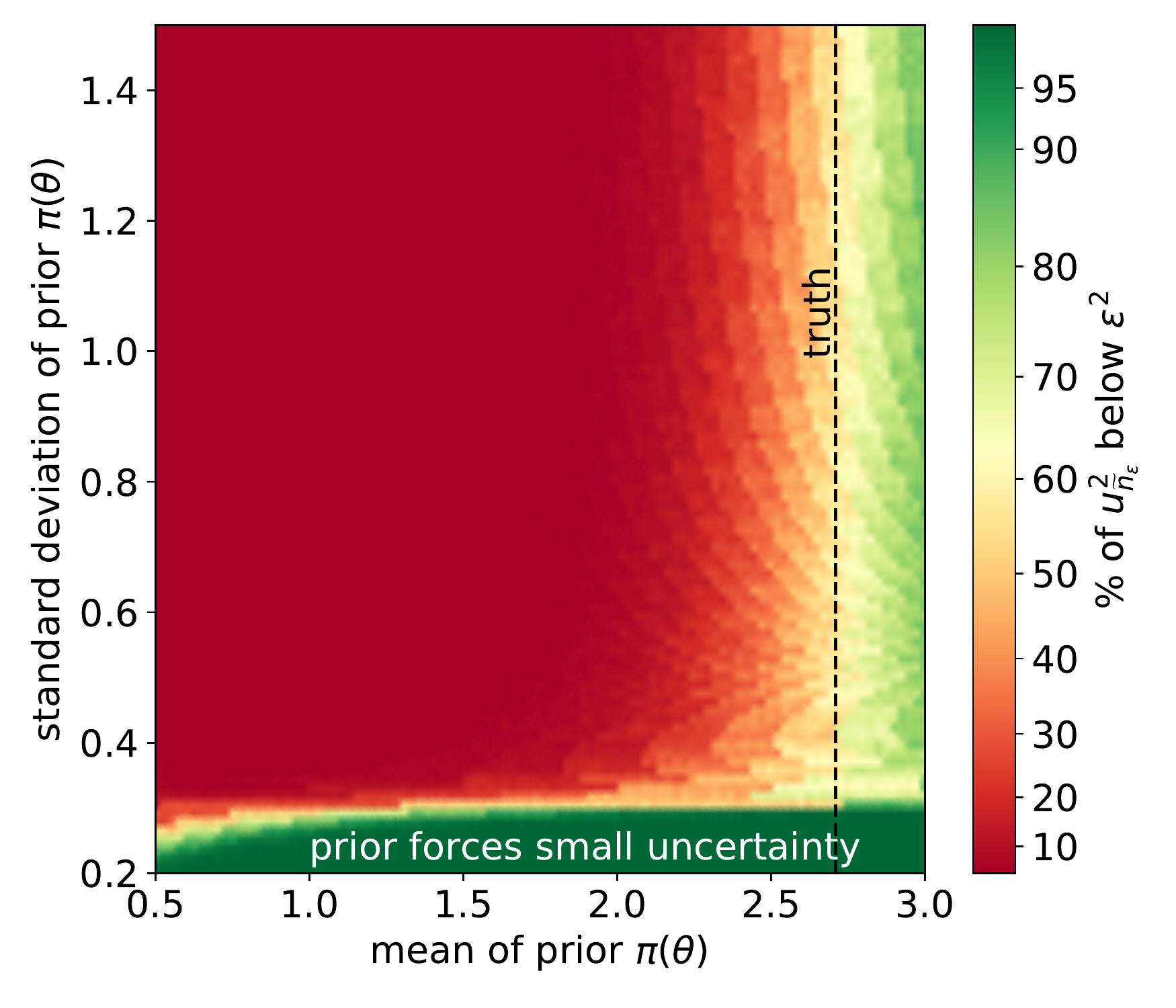}
		\caption{APVC: fraction of $u_{\widetilde{n}_\eps}^2$ below $\eps^2$}
		\label{subfig:Football_acc_APVC}
	\end{subfigure}%
~
	\begin{subfigure}[t]{0.49\textwidth}
	\centering
	\includegraphics[width=\textwidth]{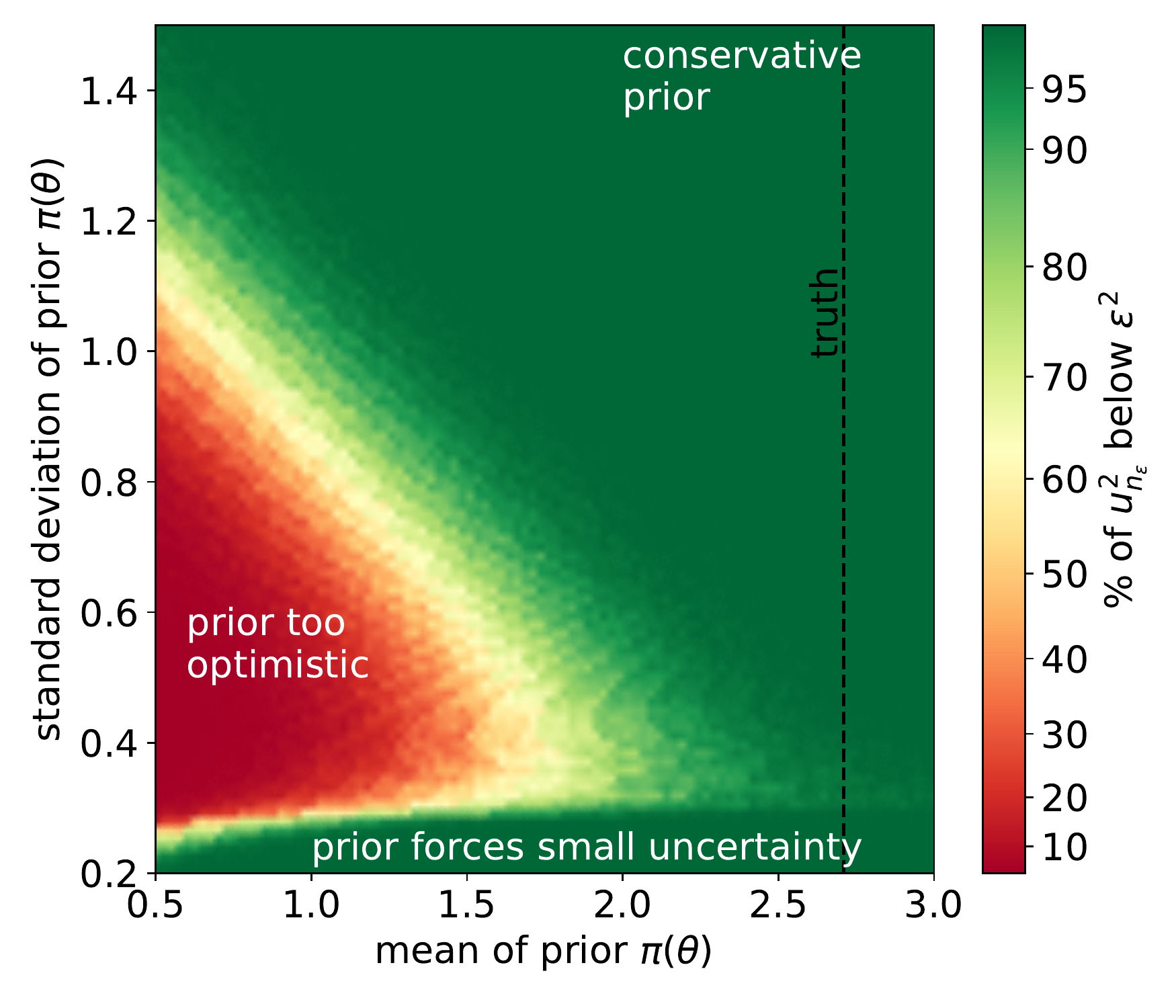}
	\caption{VPVC: fraction of $u_{n_\eps}^2$ below $\eps^2$}
		\label{subfig:Football_acc_VPVC}
	\end{subfigure}
	\caption{ \emph{left}: fraction of $u_{\widetilde{n}_\eps}^2$ below $\eps^2$ for the football goal data from Figure \ref{fig:goals} using the sample sizes $\widetilde{n}_\eps$ predicted by the APVC (Fig. \ref{subfig:Football_n_APVC}).  
\emph{right}: Fraction of $u_{n_\eps}^2$ below $\eps^2$ for the sample sizes $n_\eps$ predicted by the VPVC (Fig. \ref{subfig:Football_n_VPVC}) }
\end{figure}
As we pointed out in the introduction we therefore here use a refined criterion, which we called the VPVC and which takes the form
\begin{align}
	\label{eq:FootballVPVC}
	\overline{u_n^2}  + k\cdot \Delta u_n^2< \eps^2 \,,
\end{align}
where $\Delta u_n^2=(\Var_{\xx_n \sim m(\xx_n)}(u_n^2))^{1/2}=\frac{\alpha^{1/2}}{\beta}  \cdot \frac{n^{1/2}}{(n+\beta)^{3/2}}$ and where we choose $k=2$. The result of the VPVC is depicted in Figure \ref{subfig:Football_n_VPVC}. There are two clear differences compared to Figure \ref{subfig:Football_n_APVC}: first, the sample sizes $n_\eps$ are higher than the numbers $\widetilde{n}_\eps$ we obtained from the APVC, which was expected as we added an additional positive term $k\,\Delta u_n^2$ to the left hand side of the criterion. Second, and perhaps more important, the result of the SSD substantially increases once we increase the standard deviation of the prior. This allows us to make our sample choice more conservative for a given prior mean by increasing its standard deviation i.e our prior uncertainty about $\theta$, which is quite natural. The most conservative SSD is thus located in the right upper corner of the plot. This phenomenon for parameters that describe the variance of the data will appear again in the next subsection.

The success of the refined criterion is depicted in Figure \ref{subfig:Football_acc_VPVC}. The percentage of $u_{n_\eps}^2$ below $\eps^2$ reaches now much higher values, at many positions beyond $95\%$. This success is also more robust against deviations from the true value. Even for a mean of the prior $\pi(\theta)$ which is well below the true $\theta$ a high enough standard deviation of $\pi(\theta)$ will allow to reach a quota of $95\%$.

Figure \ref{subfig:Football_acc_VPVC} therefore splits in three areas. As in Figure \ref{subfig:Football_acc_APVC} there is a bottom part where the posterior variance is dominated by the small prior uncertainty so that $u_{n_\eps}^2 <\eps^2$ is easily satisfied even for the minimal sample size of $n_\eps=1$. As the standard deviation of $\pi(\theta)$ gets larger a too small value of the prior mean will result in a low quota: the prior was chosen too optimistic. Increasing either the mean of $\pi(\theta)$ or its standard deviation will lead however to a more conservative SSD and to a better compliance of $u_{n_\eps}^2<\eps^2$. The highest percentage can therefore be found in the upper right corner of Figure \ref{subfig:Football_acc_VPVC} and in the bottom area. 

\subsection{Nuisance parameters: normally distributed data}
\label{subsec:normal}
In this section we will consider a sample size determination for the identification of the ``typical'' song length $\mu$ of a pop song. For this purpose we will use the \emph{Million Song Dataset} from \cite{MillionSongs} that contains metadata for a million pop songs, collected in 2011. Averaging over the length of all songs in the dataset reveals that $\mu$ is around 4.17 minutes. As before, our goal is however \emph{not} to find an outmost precise value of $\mu$ but instead to identify a minimal number of songs $n$ from which we can estimate $\mu$ up to a pre-specified precision $\eps$. 
For the values of $\eps$ used in this article the sample size $n$ will turn out to be much smaller number than a million. This allows us, once more, to use the full dataset to judge the effectiveness of our SSD.

For the sake of simplicity we will make the assumption that song length is a property that's normally distributed, thereby ignoring for instance possible skewness issues that arise from the fact that a song cannot have a negative length. For $n$ songs we assume that their lengths $\xx_n=(x_1,\ldots, x_n)$ are normally distributed, i.e.
$p(\xx_n|\mu, \sigma^2) = \prod_{i=1}^n \mathcal{N}(x_i|\mu,\sigma^2)$ with mean $\mu$ and standard deviation $\sigma^2$. Both parameters $\mu$ and $\sigma^2$ are unknown. While $\mu$ is the parameter of interest, $\sigma^2$ will be considered as a nuisance parameter. We take the normal-inverse-gamma prior \cite{Fink1997}
\begin{align}
	\label{eq:musicprior}
	\pi(\mu,\sigma^2) = \mathcal{N}(\mu|\mu_0, \lambda\sigma^2) \cdot \IG(\sigma^2|\alpha,\beta) \,,
\end{align}
where $\IG$ denotes the inverse-gamma distribution and where  
$\alpha>2,\lambda,\beta>0$ and $\mu_0$ are hyperparameters.
The squared uncertainty $u_n^2$ for $\mu$ is given by the variance of the marginal posterior $\pi(\mu|\xx_n)$:
\begin{align}
	u_n^2 = \Var_{\mu \sim \pi(\mu|\xx_n)}(\mu)  = \frac{2\beta + \sum_{i=1}^n (x_i -\overline{x})^2 + \frac{n}{\lambda n_\lambda}(\overline{x}-\mu_0)^2}{n_\lambda (n + 2\alpha -2)}\,.
\end{align}
where $n_\lambda=n+\lambda^{-1}$ and $\overline{x}=\frac{1}{n}\sum_{i=1}^n x_i$.%
\begin{figure}[t]
\centering
	\begin{subfigure}[t]{0.49\textwidth}
		\centering
		\includegraphics[width=\textwidth]{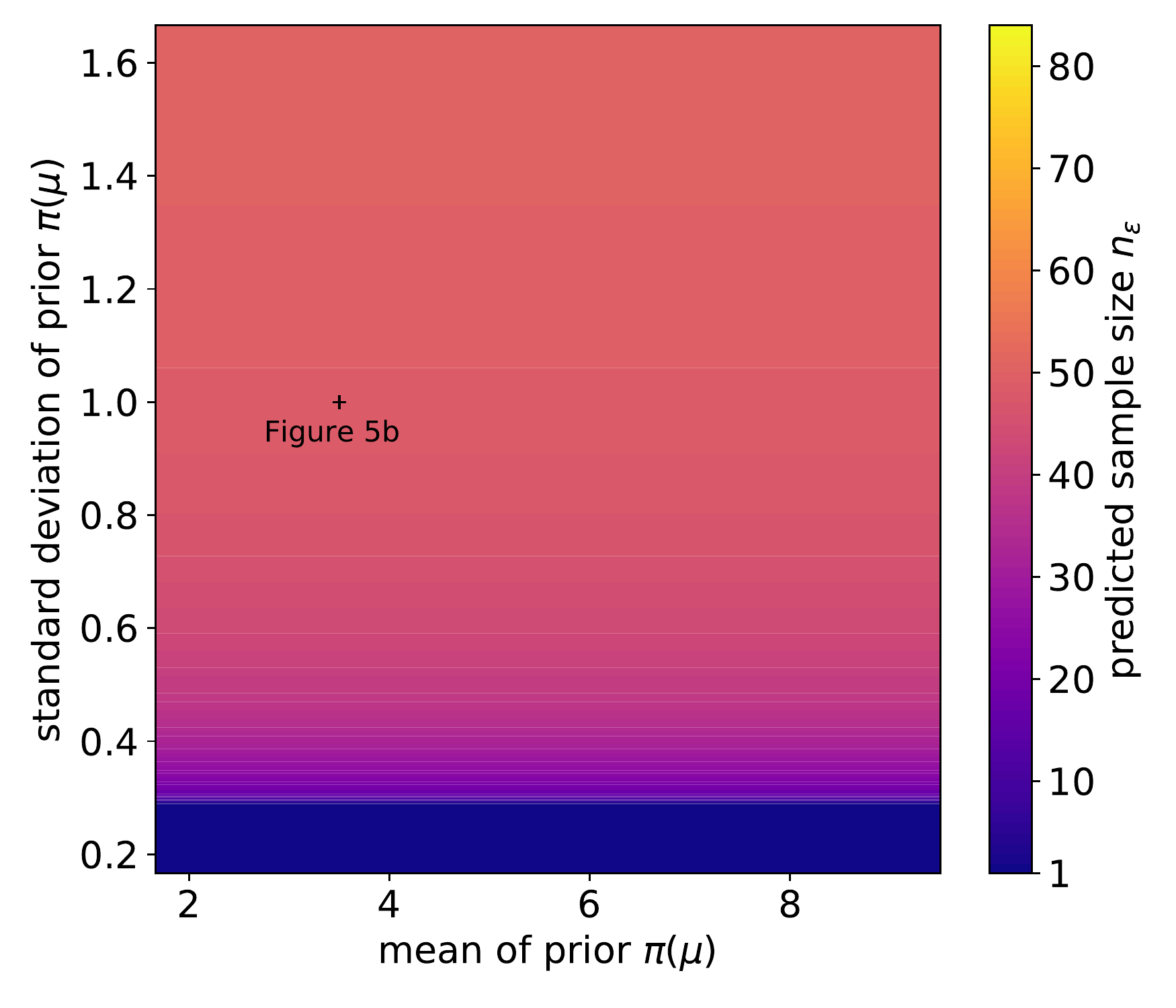}
		\caption{VPVC: SSD for varying $\pi(\mu)$}
		\label{subfig:Music_n_mu}
	\end{subfigure}%
~
	\begin{subfigure}[t]{0.49\textwidth}
		\centering
		\includegraphics[width=\textwidth]{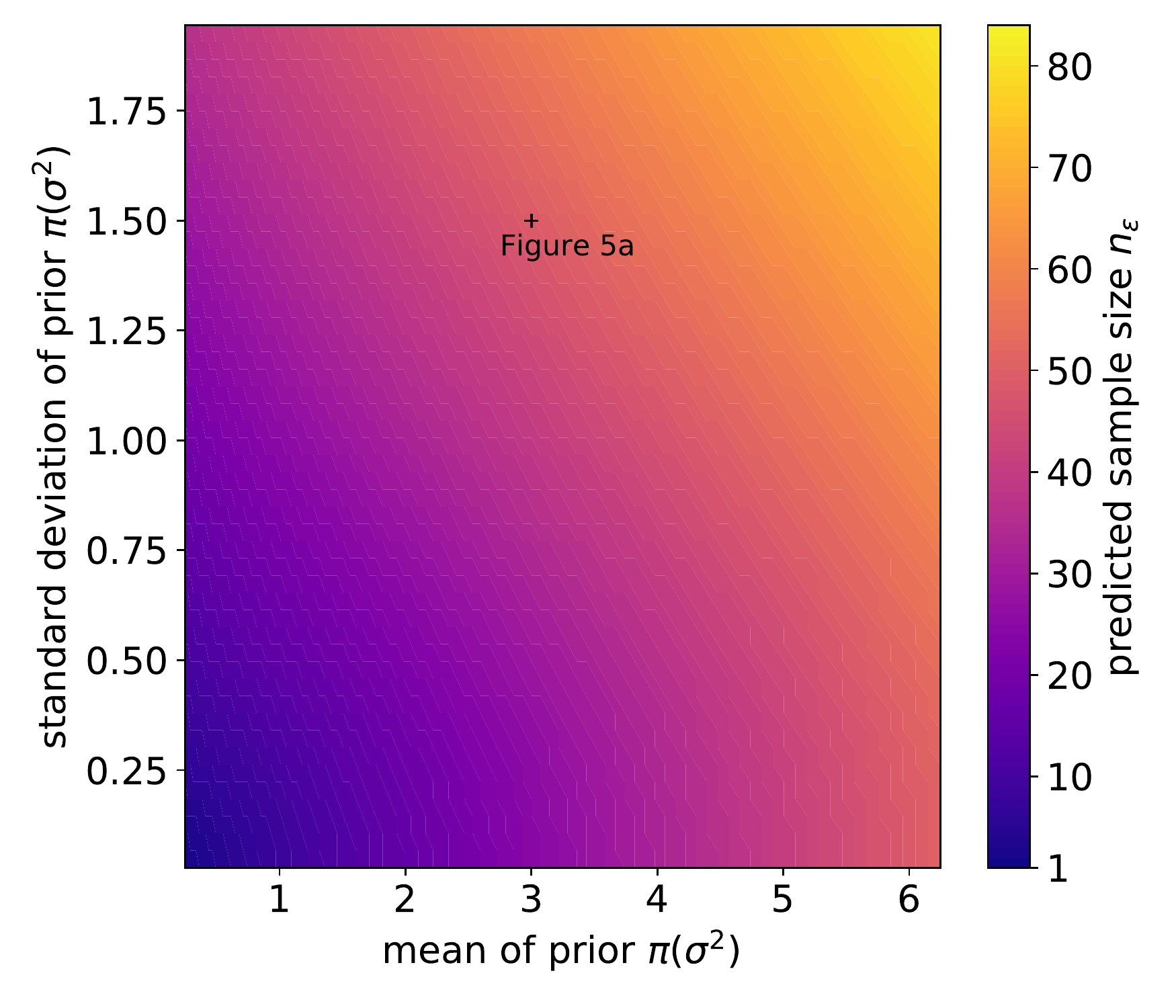}
		\caption{VPVC: SSD for varying $\pi(\sigma^2)$}
		\label{subfig:Music_n_s2}
	\end{subfigure}
	\caption{Sample sizes predicted by the VPVC for different $\pi(\mu)$ and $\pi(\sigma^2)$. \emph{left}: $\pi(\sigma^2)$ is fixed as indicated by the black cross in Figure \ref{subfig:Music_n_s2} and the sample sizes $n$ is depicted for various $\pi(\mu)$. \emph{right}: Sample sizes for various $\pi(\sigma^2)$ and a fixed $\pi(\mu)$ marked by the black cross in Figure \ref{subfig:Music_n_mu}. }
	\label{fig:Music_n}
\end{figure}
We will once more consider the VPVC, that is we choose the smallest $n=n_\eps$ such that
\begin{align}
	\overline{u_n^2} + k \cdot\Delta   u_n^2 < \eps^2 \,,
\end{align}
where $\overline{u_n^2}=\EE_{\xx_n \sim m(\xx_n)}[u_n^2]=\frac{\beta}{n_\lambda\,(\alpha-1)}$ and $\Delta u_n^2 = (\Var_{\xx_n \sim m(\xx_n)}(u_n^2))^{1/2}=\allowbreak \frac{\beta}{n_\lambda(\alpha-1)(\alpha-2)^{1/2}} \big(\frac{n}{n+2\alpha-2}\big)^{1/2}$ and where we choose again $k=2$. The sample sizes $n_\eps$ for various choices of the prior $\pi(\mu,\,\sigma^2)$ and an $\eps$ of $20\,\mathrm{sec} = 0.33 \,  \mathrm{min}$ are depicted in Figure \ref{fig:Music_n}. To visualize the impact of the prior knowledge we varied one of the marginals $\pi(\mu)$ and $\pi(\sigma^2)$ while keeping the other one fixed.  
In Figure \ref{subfig:Music_n_mu} we varied the marginal prior $\pi(\mu)$, while fixing the marginal $\pi(\sigma^2)$ to have a mean of $3.0 \, \mathrm{min}^2$ and a standard deviation of $1.5\,\mathrm{min}^2$, for comparison this $\pi(\sigma^2)$ was marked by the black cross in Figure \ref{subfig:Music_n_s2}. For Figure \ref{subfig:Music_n_s2} we fixed $\pi(\mu)$ to have a mean $3.5 \, \mathrm{min}$ and a standard deviation of $1.0 \,\mathrm{min}$ (marked by the black cross in Figure \ref{subfig:Music_n_mu}) and varied $\pi(\sigma^2)$. In particular, the two crosses in Figure \ref{subfig:Music_n_mu} and \ref{subfig:Music_n_s2} both mark positions with an equal sample size of $n=49$.
The APVC criterion from \eqref{eq:APVC} and the average coverage criterion from \cite{Adcock1988, Joseph1995} both yield a sample size of 24 for this prior.
Note, that the effect of the marginal $\pi(\mu)$ is rather minimal, once $\pi(\sigma^2)$ is kept fixed. In fact, the VPVC turns out to be independent of the hyperparameter $\mu_0$ - compare the vertical symmetry in Figure \ref{subfig:Music_n_mu}. Varying the standard deviation of $\pi(\mu)$ while keeping $\pi(\sigma^2)$ fixed will only affect the hyperparameter $\lambda$ which has only a minor influence on the SSD result. The only exception to this is the bottom area of Figure \ref{subfig:Music_n_mu} where the small standard deviation of $\pi(\mu)$ forces the posterior variance to be small and thus predicts a small sample size of 1, which is similar to a phenomenon we observed in Section \ref{subsec:Poisson}. Figure \ref{subfig:Music_n_s2} shows that the variance parameter $\sigma^2$ has a similar influence on the sample size as the Poisson parameter in Section \ref{subsec:Poisson}: increasing either the mean or the standard deviation of $\pi(\sigma^2)$ will increase the sample size so that the most conservative experimental design is located in the upper right corner. 
This behavior can again be expected since increasing the mean of $\pi(\sigma^2)$ one will expect data that has larger variability and contains less information.
In contrast to Figure \ref{subfig:Football_n_VPVC} and Figure \ref{subfig:Music_n_mu} there is no distinct bottom area of minimal sample sizes.  

\begin{figure}[t]
\centering
	\begin{subfigure}[t]{0.49\textwidth}
		\centering
		\includegraphics[width=\textwidth]{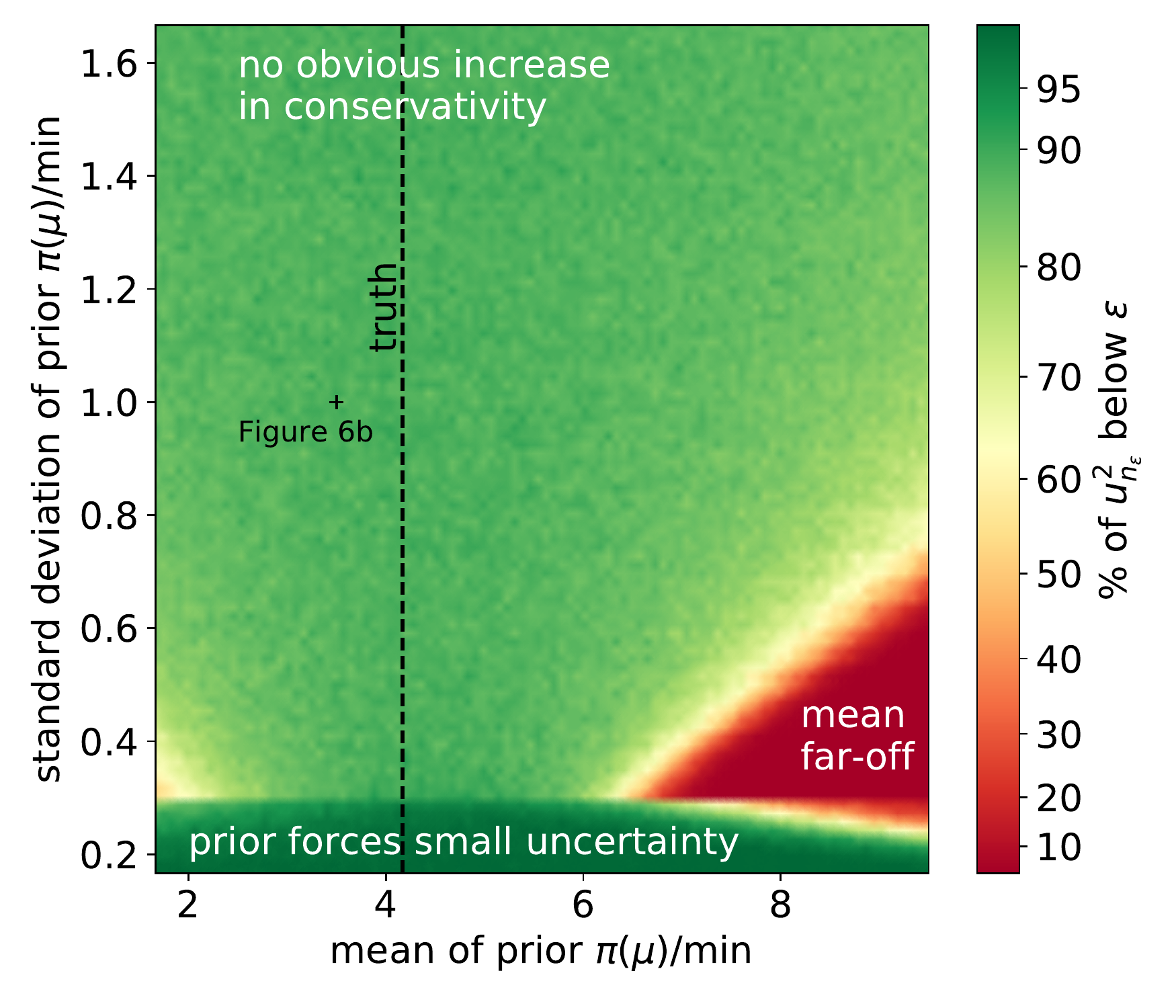}
		\caption{VPVC: SSD for varying $\pi(\mu)$}
		\label{subfig:Music_acc_mu}
	\end{subfigure}%
~
	\begin{subfigure}[t]{0.49\textwidth}
		\centering
		\includegraphics[width=\textwidth]{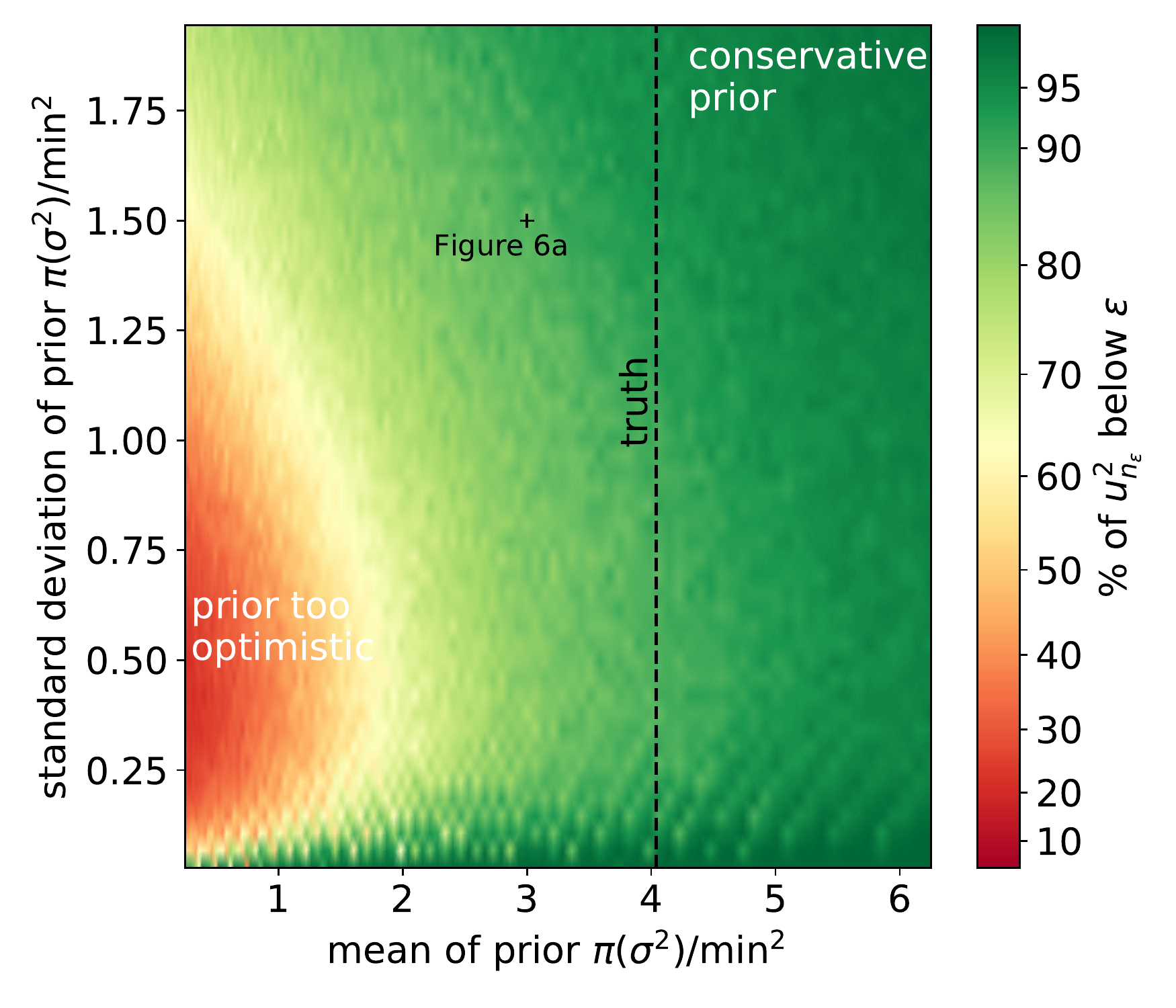}
		\caption{VPVC: SSD for varying $\pi(\sigma^2)$}
		\label{subfig:Music_acc_s2}
	\end{subfigure}
	\caption{Percentage of $u_{n_\eps}^2 < \eps^2$ for the sample sizes $n_\eps$ from Figure \ref{fig:Music_n} within the \emph{Million song dataset} \cite{MillionSongs}. Axes of both plots and the positions of the black crosses are the same as in Figure \ref{fig:Music_n}. The ``true'' value within the full dataset is depicted by the dashed lines.}
	\label{fig:Music_acc}
\end{figure}

As in Section \ref{subsec:Poisson} we can use the full dataset to judge the success of our sample size planning. 
Figure \ref{fig:Music_acc} displays, for the same priors as in Figure \ref{fig:Music_n} the proportion of $u_{n_\eps}^2$ below $\eps^2$ for $u_{n_\eps}^2$ computed from random samples of size $n_\eps$, drawn from the Million Song Dataset. The left plot, Figure \ref{subfig:Music_acc_mu},  reveals that the choice of the marginal $\pi(\mu)$ has mostly an impact on the success of the SSD if its mean is far-off from the true parameter. Moreover, a very precise prior knowledge about $\mu$, will result in a high proportion of $u_{n_\eps}^2 < \eps^2$, which is similar to the observations we made for the Poisson example in Section \ref{subsec:Poisson}. The behavior of the percentage for which $u_{n_\eps}^2 < \eps^2$ when varying $\pi(\sigma^2)$ is similar to what we have observed for the Poisson parameter in Section \ref{subsec:Poisson}. For small values of the mean of $\pi(\sigma^2)$ the quota drops, which can be cured by increasing the standard deviation of $\pi(\sigma^2)$. The highest percentage can be found in the upper right corner.  For the prior marked by the black crosses in Figure \ref{fig:Music_acc}, the same as the one marked in Figure \ref{fig:Music_n}, the quota is around $88.7\%$ while the sample size predicted by the APVC and average coverage criterion only yields a percentage of around $49.9\%$. 
The bottom area of Figure \ref{subfig:Music_acc_s2}, where the posterior variance is heavily influenced by the prior, is not as distinct as in Figure \ref{subfig:Football_acc_VPVC}, this role seems to be played in this context by $\mu$ - compare Figure \ref{subfig:Music_acc_mu}.

Let us shortly summarize the most important observations we have made concerning SSD based on the prior predictive $m(\xx_n)$ via using the VPVC \eqref{eq:VPVC}. There are two strategies to have a good chance of achieving $u_n^2 < \eps^2$:

\begin{enumerate}
	\item A precise prior knowledge about the parameter of interest. This will result in small sample sizes
	\item A conservative estimate of the parameter that determines the variation of the data. This will result in larger sample sizes. 
		We have observed two different ways to achieve this.
		\begin{itemize}
			\item Make a conservative (that is rather large) guess for the mean of the prior marginal w.r.t this parameter.
			\item Choose a large enough standard deviation for this marginal prior. In order for this to work it is not sufficient to base the SSD on the average $\overline{u_n^2}$ only but higher moments as in the VPVC have to be taken into account.
		\end{itemize}
\end{enumerate}

In the Poisson example from Section \ref{subsec:Poisson} both strategies 1 and 2 concern the same parameter $\theta$. For the case of normally distributed data  we have observed that strategy 1 concerns the parameter of interest $\mu$ while to follow strategy 2 the parameter $\sigma^2$, that describes the variance of the data, was central.

\section{Behavior for small $\eps$}
\label{sec:asymptotics}

We have seen in Section \ref{sec:SSD} that the determination of the sample size via the VPVC can be successful, provided the prior knowledge is either precise enough or sufficiently conservative. Being ``conservative'' is of course a rather vague quality, but in this section we want to answer what happens if $\eps$ from \eqref{eq:VPVC} becomes smaller. Will it become easier or harder to be conservative? In Theorem \ref{thm:conservative} we show that as $\eps\rightarrow 0$ the VPVC tends to satisfy \eqref{eq:SSD_Goal} perfectly provided $k$ is bigger then some threshold $k_\ast$. In other words for such a $k$ the VPVC has a tendency to become conservative. 
If $k$ is smaller than $k_\ast$ on the other hand we will show that, asymptotically, the uncertainty will almost surely be above the expression used for the VPVC. This phase transition with respect to $k$, compare Figure \ref{fig:phase_transition} in the appendix, gives some inside on SSD based on the prior predictive in the small $\eps$ regime. We will give some ideas how to get an upper bound on $k_\ast$. Another result of this section is Lemma \ref{lem:n_eps}, where we show an asymptotic formula for the sample size $n_\eps$ determined by the VPVC.

First let us fix our setup. Similar as in Section \ref{sec:SSD} we will assume that there is univariate parameter of interest $\theta_0$ and a, possibly empty or multivariate, nuisance parameter $\theta'$ so that the total parameter that determines our sampling distribution $p(\xx_n |\theta)=p(x_1,\ldots, x_n|\theta)$ is given by
\begin{align}
	\theta = (\theta_0, \theta') \,.
\end{align}
We will denote the Fisher information matrix of the single sample distribution $p(x_1|\theta)$ by $I_\theta$ and its first component (the one linked to $\theta_0$) by $I_{\theta_0}$. Note, that since $\theta_0$ is univariate $I_{\theta_0}$ is a non-negative number. The knowledge about $\theta$ is described by the prior $\pi(\theta)$ with marginals $\pi(\theta_0)$ and $\pi(\theta')$. We further fix a value $\theta_{\mathrm{true}} =(\theta_{\mathrm{true}, 0}, \theta_{\mathrm{true}}')$ in the domain of $\pi(\theta)$ that we treat as the true parameter. 
In practice $\theta_{\mathrm{true}}$ will of course be unknown.
Given some data $\xx_n = (x_1,\ldots, x_n)$ we introduce as above the squared uncertainty by
$ u_n^2 = \Var_{\theta_0 \sim \pi(\theta_0|\xx_n)}(\theta_0)$
as well as $\overline{u_n^2}=\EE_{\xx_n \sim m(\xx_n)}[u_n^2]$, $\Delta u_n^2 = (\Var_{\xx_n \sim m(\xx_n)}(u_n^2))^{1/2}$ with $m(\xx_n) = \int p(\xx_n|\theta) \pi(\theta) \dd \theta$ and formulate the VPVC for $\eps>0$ as
\begin{align}
	\label{eq:AsVPVC}
	\overline{u_n^2} + k\,\Delta u_n^2 < \eps^2  \,.
\end{align}
The smallest $n$ satisfying \eqref{eq:AsVPVC} will be called once more $n_\eps$. 

\begin{assumption}
	\label{ass:bernstein}
	We assume that for $\pi(\theta)$ and $\theta_{\mathrm{true}}$ it holds:
	\begin{enumerate}
		\item \emph{Fisher information sufficiently regular}: The Fisher information $I_{\theta_{0}}$ is strictly positive for $\theta=\theta_{\mathrm{true}}$ and almost every $\theta$ in the domain of $\pi(\theta)$. Moreover we assume that the second moment $\EE_{\theta \sim \pi(\theta)}[I_{\theta_0}^{-2}]$ exists and that $\Var_{\theta \sim \pi(\theta)}(I_{\theta_0}^{-1})>0$.
		\item \emph{B.-v.-Mises limit holds in $L^2$}: The quantity $n\cdot u_n^2 $ converges in $L^2$ against $I_{\theta_0}^{-1}$ conditional under $\theta=\theta_{\mathrm{true}}$ and $\pi(\theta)$ in the sense that:
			\begin{align}
				\label{eq:BvM_L2}
				\begin{aligned}
					&\lim_{n\rightarrow\infty} \EE_{\xx_n \sim p(\xx_n | \theta_{\mathrm{true}})}[|n\cdot u_n^2 - I_{\theta_{\mathrm{true}, 0}}^{-1}|^2] = 0 \mbox { and  }  \\
				&\lim_{n\rightarrow\infty} \EE_{\theta \sim \pi(\theta)}[\EE_{\xx_n \sim p(\xx_n|\theta)}[|n\cdot u_n^2 - I_{\theta_0}^{-1}|^2]] = 0 \,.
				\end{aligned}
			\end{align}
		\item \emph{No trivial SSD}: For all $n$ we have $\overline{u_n^2}>0$.
	\end{enumerate}
\end{assumption}

From these assumptions point 2 is probably the one that needs the most explanation. The Bernstein-von-Mises theorem \cite{Vaart2000} indicates that under relatively mild conditions, and conditional on $\theta$, the posterior distribution is asymptotically normal with variance $\frac{1}{n} I_{\theta_0}^{-1}$, which motivates \eqref{eq:BvM_L2}. When considering variances, as in this article, it is then quite natural to require convergence in $L^2$, while the B.-v.-Mises theorem only ensures convergence in probability. 
While standard methods, such as Vitali's convergence theorem, could be applied to provide the first convergence \eqref{eq:BvM_L2} these seems more involved for the second limit as it is in a rather non-standard form. 
Examining this point in more detail would stray far away from the scope of this article and could be subject to future research. 
However, we found that both identities \eqref{eq:BvM_L2} were rather easy to check for many standard cases such as the ones from Section \ref{sec:SSD} or a Bernoulli distribution with a Beta prior. 

Under Assumptions \ref{ass:bernstein} the posterior variance is, conditional on $\theta$, asymptotically proportional to $I_{\theta_0}^{-1}$. The following quantity therefore describes the variation of the asymptotic posterior variance given the prior knowledge $\pi(\theta)$:  
\begin{align}
	\label{eq:gamma}
	\gamma = \frac{(\Var_{\theta \sim \pi(\theta)}(I_{\theta_0}^{-1}))^{1/2}}{\EE_{\theta \sim \pi(\theta)}[I_{\theta_0}^{-1}]} \,.
\end{align}

Let us make these considerations more precise.

\begin{lemma} \label{lem:asymptotics} 
Define the coefficient of variation under the prior predictive $m(\xx_n)$ and under $\theta_{\mathrm{true}}$ as 
\begin{align*}
	c_{n} = \frac{\Delta u_n^2}{\overline{u_n^2}} \qquad  \mbox{ and } \qquad
	c_{\mathrm{true}, n} = \frac{(\Var_{\xx_n \sim p(\xx_n |\theta_{\mathrm{true}})}(u_n^2))^{1/2}}{\EE_{\xx_n \sim p(\xx_n|\theta_{\mathrm{true}})}[u_n^2]} \,. 
\end{align*} Under Assumptions \ref{ass:bernstein} we then have \begin{align*} \lim_{n\rightarrow \infty} c_{n} = \gamma \mbox{ and } \lim_{n\rightarrow \infty} c_{\mathrm{true}, n} = 0 \,.
\end{align*}
\end{lemma}
\begin{remark}
The limit of $c_{\mathrm{true}, n}$ is a special case of the one for $c_{n}$ by taking the delta distribution $\delta_{\theta_{\mathrm{true}}}$, centered at $\theta_{\mathrm{true}}$, as a prior. 
\end{remark}
\begin{proof} Let us first consider the limit of $c_n$. Note that the second condition of point 2 in Assumption \ref{ass:bernstein} states that $X_n := n\, u_n^2$ converges to $I_{\theta_0}^{-1}$ in $L^2$ with respect to the law $p(\theta, \xx_n) = \pi(\theta)\cdot p(\xx_n|\theta)$. 	Since $X_n$ is not dependent on $\theta$, $I_{\theta_0}$ is not dependent on $\xx_n$ and the marginals of $p(\theta,\xx_n)$ are $\pi(\theta)$ and $m(\xx_n)$ we conclude that \begin{align*}
	\Var_{\xx_n \sim m(\xx_n)} (X_n) &= \Var_{(\theta, \xx_n) \sim p(\theta,\xx_n)} (X_n)  \\
		&\rightarrow \Var_{(\theta, \xx_n) \sim p(\theta,\xx_n)} (I_{\theta_0}^{-1}) = \Var_{\theta \sim \pi(\theta)} (I_{\theta_0}^{-1})
	\end{align*}
and similar $\EE_{\xx_n \sim m(\xx_n)}[X_n] \rightarrow \EE_{\theta \sim \pi(\theta)}[I_{\theta_0}^{-1}]$. From this we obtain the claimed limit of $c_n$:
\begin{align*}
	c_n  	
	\overset{X_n = n\,u_n^2}{=}
	\frac{(\Var_{\xx_n \sim m(\xx_n)} (X_n))^{1/2}}{\EE_{\xx_n \sim m(\xx_n)}[X_n]}
	\rightarrow
	\frac{(\Var_{\theta\sim \pi(\theta)}(I_{\theta_0}^{-1}))^{1/2}}{\EE_{\theta\sim \pi(\theta)}[I_{\theta_0}^{-1}]} =
	\gamma\,.
\end{align*}
For the second part of the claim we use that by the first condition in point 2 of Assumption \ref{ass:bernstein} we have $L^2$ convergence of $X_n$ conditional on $\theta_{\mathrm{true}}$ from which we obtain indeed
\begin{align*}
	c_{\mathrm{true}, n} = \frac{(\Var_{\xx_n \sim p(\xx_n|\theta_{\mathrm{true}})}(X_n))^{1/2}}{\EE_{\xx_n \sim p(\xx_n|\theta_{\mathrm{true}})}[X_n]}  \rightarrow \frac{(\Var_{\xx_n \sim p(\xx_n|\theta_{\mathrm{true}})}(I_{\theta_{\mathrm{true}, 0}}^{-1}))^{1/2}}{\EE_{\xx_n \sim p(\xx_n|\theta_{\mathrm{true}})}[I_{\theta_{\mathrm{true}, 0}}^{-1}]} = 0 \,,
\end{align*}
where we used on the right hand side that $I_{\theta_{\mathrm{true}, 0}}$ does not depend on $\xx_n$.
\end{proof}

 We see from Lemma \ref{lem:asymptotics} that the variation of $u_n^2$ under the marginal $m(\xx_n)$ is behaving in a different manner as under the true parameter $\theta_{\mathrm{true}}$. The object $\Delta u_n^2 = (\Var_{\xx_n \sim m(\xx_n)}(u_n^2))^{1/2}$ decays at the same order as $\overline{u_n^2}=\EE_{\xx_n \sim m(\xx_n)}[u_n^2]$ whereas, being conditional on a parameter, the standard deviation $(\Var_{\xx_n \sim p(\xx_n|\theta_{\mathrm{true}})}(u_n^2))^{1/2}$ decays faster than $\EE_{\xx_n \sim p(\xx_n|\theta_{\mathrm{true}})}[u_n^2]$. The distributions $m(\xx_n)$ and $p(\xx_n|\theta_{\mathrm{true}})$ result in quite different asymptotic behavior for the moments of the posterior variance. 
We will see in Theorem \ref{thm:conservative} below that these different asymptotic properties will provoke some phase transition with respect to $k$ in the limit $\eps\rightarrow 0$.  

As a first consequence of the observations from Lemma \ref{lem:asymptotics} we will now see that the variance term $k\,\Delta u_n^2$ in the VPVC will strongly affect the sample size even as $\eps \rightarrow 0$. 

\begin{lemma}[VPVC sample size for small $\eps$]
	\label{lem:n_eps}
	Provided Assumption \ref{ass:bernstein} is true, we have
	\begin{align}
		\lim_{\eps \rightarrow 0} \frac{n_\eps}{\eps^{-2}\cdot (1+k\,\gamma)\EE_{\theta \sim \pi(\theta)}[I_{\theta_0}^{-1}]} = 1
	\end{align}
\end{lemma}
\begin{proof}
	We have used already in the first part of the proof of Lemma \ref{lem:asymptotics} that $n\overline{u_n^2} \rightarrow \EE_{\theta\sim \pi(\theta)}[I_{\theta_0}^{-1}]$ so that applying in addition the convergence $c_n\rightarrow \gamma$ from Lemma \ref{lem:asymptotics} we arrive at 
	\begin{align}
		\begin{aligned}
			s_n &:= n\,  (\overline{u_n^2} + k\, \Delta u_n^2) = n\overline{u^2_n} (1+k\, c_n) 
			&\overset{n\rightarrow \infty}{\longrightarrow} (1+k\,\gamma)\EE_{\theta\sim \pi(\theta)}[I_{\theta_0}^{-1}]=: s_{\infty} \,.
		\end{aligned}
	\label{eq:s_n_convergence}
	\end{align}
	Next, observe that since $\overline{u_n^2}>0$ by point 3 of Assumption \ref{ass:bernstein} this implies that we can find a constant $c>0$ such that for any $n$ we have $\frac{c}{n}<\overline{u_n^2}+k\Delta u_n^2$, from which we get $n_\eps > c \cdot \eps^{-2}$ and in particular
	\begin{align}
		\label{eq:n_eps_diverges}
		\lim_{\eps \rightarrow 0} n_\eps = \infty \,.
	\end{align}
	Let us rewrite the claim of the lemma as
	\begin{align}
		\lim_{\eps \rightarrow 0} \frac{n_\eps}{\eps^{-2} \cdot s_\infty}  =1\,.
		\label{eq:claim_lemma}
	\end{align}
	From \eqref{eq:s_n_convergence}, and the choice of $n_\eps$, we already have one inequality of \eqref{eq:claim_lemma}:
	\begin{align*}
		\liminf_{\eps \rightarrow 0} \frac{n_\eps}{\eps^{-2} s_{\infty}} \overset{\mbox{\tiny def. of  } n_\eps}{\geq} \liminf_{\eps\rightarrow 0} 
		\frac{n_\eps \cdot (\overline{u^2_{n_\eps}}+k\,\Delta u^2_{n_\eps})}{s_\infty} 
		=\liminf_{\eps \rightarrow 0} \frac{s_{n_\eps}}{s_\infty} \overset{\mbox{\tiny \eqref{eq:s_n_convergence},\,\eqref{eq:n_eps_diverges}}}{=} 1 \,.
	\end{align*}
	To show the opposite inequality we argue by contradiction. Suppose that there is a $\delta>0$ such that
	\begin{align}
		\label{eq:absurd_statement}
		\limsup_{\eps \rightarrow 0} \frac{n_\eps}{\eps^{-2} \cdot s_\infty}  \geq  1+2\,\delta \,.
	\end{align}
	Define $m_\eps=\lfloor s_\infty \eps^{-2}\rfloor$, the largest integer less equal $s_\infty \eps^{-2}$, and note that for $\eps$ small enough there is an integer $n_\eps' $ such that the following nested inequality is true
	\begin{align}
		\label{eq:nested_ineq}
		1 \leq m_\eps <(1+\delta/2) (1+m_\eps) \leq n_\eps' \leq m_\eps (1+\delta) \overset{\mbox{\tiny \eqref{eq:absurd_statement}}}{<} n_\eps \,.
	\end{align}
	Choose further $\eps$ small enough such that for any $n > m_\eps$ we have $s_{n}/s_\infty < (1+\delta/2)$. But then the following inequality holds:
	\begin{align*}
		\overline{u^2_{n_\eps'}} +k\cdot \Delta u^2_{n_\eps'} = \frac{s_{n_\eps'}}{n_\eps'}
		\overset{n_\eps'>m_\eps}{<} \frac{(1+\delta/2)\cdot s_\infty}{n_\eps'} \overset{\eqref{eq:nested_ineq}}{\leq} 
		\frac{s_\infty}{1 + \lfloor s_\infty \eps^{-2} \rfloor} \leq \eps^2 \,.
	\end{align*}
	By \eqref{eq:nested_ineq} we have thus found a smaller integer $1 \leq n_\eps' < n_\eps$ that satisfies the VPVC \eqref{eq:AsVPVC}, which contradicts the choice of $n_\eps$. 
\end{proof}

For the setup considered in Section \ref{subsec:normal} we sketched the convergence of the object from Lemma \ref{lem:n_eps} for various priors in Figure \ref{fig:asymptotic_formula}. The formula from Lemma \ref{lem:n_eps} allows for a few interesting observations. 

\begin{wrapfigure}{l}{0.5\textwidth}
		\centering
		\includegraphics[width=0.48\textwidth]{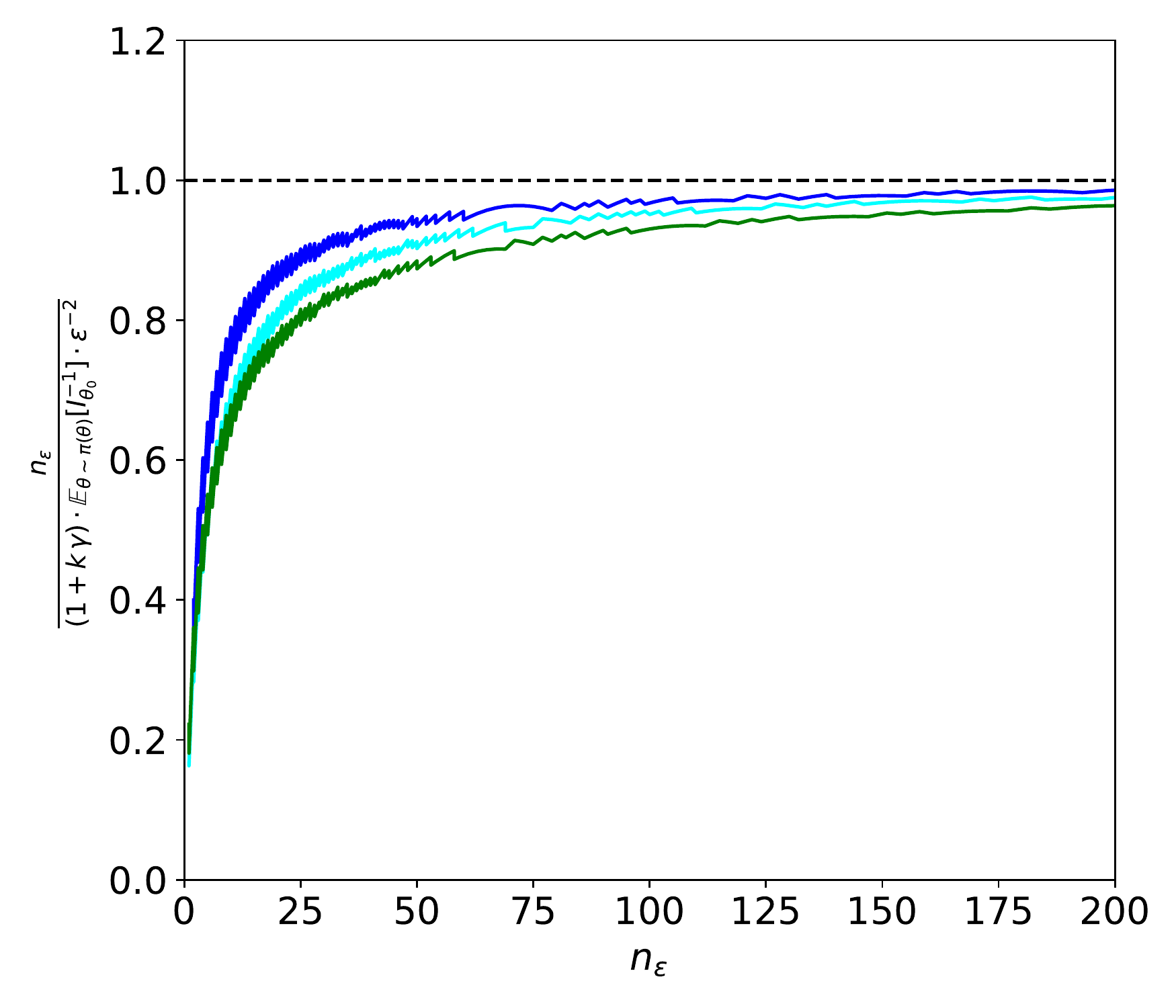}
		\caption{Convergence of the expression from Lemma \ref{lem:n_eps} for the setup from Section \ref{subsec:normal}, $k=2$, various priors $\pi(\sigma^2)$ and the fixed prior $\pi(\mu)$ marked in Figure \ref{subfig:Music_n_mu} with mean 3.5$\,\mathrm{min}$ and standard deviation 1.0$\,\mathrm{min}$. Choices for $\pi(\sigma^2)$, all in $\mathrm{min}^2$: \emph{blue}: $2.0\pm 1.5$, \emph{cyan}: $3.0\pm 1.5$ (same as marked in Fig. \ref{subfig:Music_n_s2}), \emph{green}: $3.0\pm 0.75$.   }
		\label{fig:asymptotic_formula}
\end{wrapfigure}
First, we see that asymptotically the prior $\pi(\theta)$ has an impact on the sample size through $\EE_{\theta \sim \pi(\theta)}[I_{\theta_0}^{-1}]$ (the expectation of the rescaled B.-v.-Mises limit of the posterior variance) and through the coefficient of variation $\gamma$. The higher the latter the more influence will the choice of $k$ have on the sample size. Moreover, comparing the sample sizes $n_\eps$ from the VPVC and $\tilde{n}_\eps$ from the APVC ($k=0$) we see that taking the variance into account will increase the sample size by a factor of $\frac{n_\eps}{\tilde{n}_\eps} \simeq 1+k\cdot \gamma$. Finally, note that we do not need any explicit expression for $u_n^2$ when using the asymptotic expression of $n_\eps$. Given the Fisher information $I_{\theta_0}$ of the sampling distribution and any prior $\pi(\theta)$ we can directly compute the sample size for small $\eps$ without any need of computing the posterior distribution.

We will now turn to the main result of this section.
In Section \ref{sec:SSD} we have evaluated the ``success'' of the SSD by evaluating how often we will have $u_{n_\eps}^2 <\eps^2$ on the actual dataset, recall for instance Figure \ref{subfig:Football_acc_VPVC}. In the following we want to do this for $\eps\rightarrow 0$ in a purely generic setup by assuming that the data $\xx_n$ follows the distribution $p(\xx_n|\theta_{\mathrm{true}})$ with the true parameter $\theta_{\mathrm{true}}$. We have already observed in Lemma \ref{lem:asymptotics} that conditioning on a parameter like $\theta_{\mathrm{true}}$ leads to a different asymptotic behavior than considering the marginal $m(\xx_n)$.  The following theorem shows the consequence of this disparity concerning the success of the SSD.

\begin{theorem}[VPVC becomes conservative for large $k$]
	\label{thm:conservative}
	Assume that Assumptions \ref{ass:bernstein} hold and define
	\begin{align}
		\label{eq:k_ast}
		k_\ast = k_\ast (\theta_{\mathrm{true}}) := \frac{\max(I_{\theta_{\mathrm{true},0}}^{-1}-\EE_{\theta \sim \pi(\theta)}[I_{\theta_0}^{-1}],0)}{(\Var_{\theta \sim \pi(\theta)}(I_{\theta_0}^{-1}))^{1/2}} \,.
	\end{align}
	We then have for any $k>k_\ast$
	\begin{align}
		\label{eq:variance_is_conservative}
		\lim_{n\rightarrow \infty} \mathbb{P}_{\xx_{n} \sim p(\xx_{n}|\theta_{\mathrm{true}})}\left(u_n^2 > \overline{u_n^2} + k\,\Delta u_n^2 \right) = 0 \,. 
	\end{align}
	In particular, we have for any SSD build on the VPVC \eqref{eq:AsVPVC} with such a $k$
	\begin{align}
		\label{eq:VPVC_conservative}
		\lim_{\eps \rightarrow 0} \mathbb{P}_{\xx_{n_\eps} \sim p(\xx_{n_\eps}|\theta_{\mathrm{true}})}\left(u_{n_\eps}^2 \geq \eps^2\right) = 0 \,.
	\end{align}
	Moreover, $k_\ast$ is optimal in the sense that for any $0<k<k_\ast$ we have 
	\begin{align}
		\label{eq:careful}
		\lim_{n\rightarrow \infty} \mathbb{P}_{\xx_n \sim p(\xx_n|\theta_{\mathrm{true}})}\left(u_n^2 < \overline{u_n^2} + k \Delta u_n^2 \right) = 0 \,.
	\end{align}
\end{theorem}

\begin{proof}
	We introduce $\rho=\frac{I_{\theta_{\mathrm{true},0}}^{-1}}{\EE_{\theta \sim \pi(\theta)}[I_{\theta_0}^{-1}]}-1$ so that $k_\ast$ can be written as $k_\ast=\max(\rho,0)/\gamma$ with $\gamma$ as in Lemma \ref{lem:asymptotics}. Similar as in the proof of Lemma \ref{lem:asymptotics} we see that
	\begin{align}
		\label{eq:rho}
		\rho = \lim_{n\rightarrow 0} \frac{\EE_{\xx_n \sim p(\xx_n |\theta_{\mathrm{true}})}[u_n^2]}{\overline{u_n^2}} -1 \,.
	\end{align}
	Let us first show \eqref{eq:variance_is_conservative}. We can rewrite the inequality inside the probability as
\begin{align}
	\label{eq:ChebyProb}
	u_n^2 - \EE_{\xx_n \sim p(\xx_n |\theta_{\mathrm{true}})}[u_n^2]> \overline{u_n^2}- \EE_{\xx_n \sim p(\xx_n |\theta_{\mathrm{true}})}[u_n^2] + k \Delta u_n^2  \,.
\end{align}
	The expression on the right hand side becomes finally positive when $n$ is high enough. Indeed, we can we reshape it, using point 3 of Assumption \ref{ass:bernstein} and the object $c_n$ from Lemma \ref{lem:asymptotics}, as 
\begin{align*}
	\overline{u_n^2} \cdot \left( 1- \frac{\EE_{\xx_n\sim p(\xx_n|\theta_{\mathrm{true}})}[u_n^2]}{\overline{u_n^2}}+ k c_n \right) =: \overline{u_n^2} \cdot h_n
\end{align*}
	where $h_n$ satisfies $\lim_{n\rightarrow \infty} h_n = -\rho + k\cdot \gamma>0$ due to Lemma \ref{lem:asymptotics}, \eqref{eq:rho} and the choice $k>k_{\ast}=\max(\rho,0)/\gamma$. 
	For $n$ large enough such that the right hand side of \eqref{eq:ChebyProb} is positive, we can apply Chebyshev's inequality, which yields 
\begin{align*}
	&\limsup_{n\rightarrow \infty} \mathbb{P}_{\xx_n \sim p(\xx_n|\theta_{\mathrm{true}})}\left(u_n^2 > \overline{u_n^2}+ k \Delta u_n^2 \right) \\
	&\leq \limsup_{n\rightarrow \infty} \frac{\Var_{\xx_n \sim p(\xx_n|\theta_{\mathrm{true}})}(u_n^2)}{(\overline{u_n^2}- \EE_{\xx_n \sim p(\xx_n |\theta_{\mathrm{true}})}[u_n^2] + k \Delta u_n^2)^2 }  \\
	& = \limsup_{n\rightarrow \infty} c_{\mathrm{true},n}^2 \cdot \frac{\EE_{\xx_n \sim p(\xx_n|\theta_{\mathrm{true}})}[u_n^2]^2}{\overline{u_n^2}\cdot h_n^2} \leq \lim_{n\rightarrow \infty} c_{\mathrm{true},n}^2 \cdot \left(\frac{1+\rho}{-\rho+k\gamma}\right)^2=0 \,,
\end{align*}
	where we used once more \eqref{eq:rho} and that we know from Lemma \ref{lem:asymptotics} that $c_{\mathrm{true}, n}$ converges to 0. From \eqref{eq:variance_is_conservative} we immediately obtain \eqref{eq:VPVC_conservative} by the definition of $n_\eps$ and the fact that $n_\eps \rightarrow \infty$, due to Lemma \ref{lem:n_eps}.  To show \eqref{eq:careful} we reshape the expression inside the parentheses of \eqref{eq:careful} into 
	\begin{align}
		\EE_{\xx_n \sim p(\xx_n|\theta_{\mathrm{true}})}[u_n^2] - u_n^2 > \EE_{\xx_n \sim p(\xx_n|\theta_{\mathrm{true}})}[u_n^2] - (\overline{u_n^2} + k \Delta u_n^2 ) \,.
	\end{align}
	The rest of the argument then follows along the lines of the first part of the proof.
\end{proof}

Theorem \ref{thm:conservative} reveals some sort of phase transition. For a $k$ smaller than some threshold $k_\ast$ the squared uncertainty $u_n^2$ will asymptotically be above $u_n^2 + k\Delta u_n^2$, while for $k>k_\ast$ we will have $u_n^2 < \overline{u_n^2} + k \Delta u_n^2$ with probability almost 1 for large $n $ (or small $\eps$). 
This effect is shown in Figure \ref{fig:theorem_figure} for the Football goal data from section \ref{subsec:Poisson} and a fixed prior. 

\begin{wrapfigure}{l}{0.5\textwidth}
		\centering
		\includegraphics[width=0.48\textwidth]{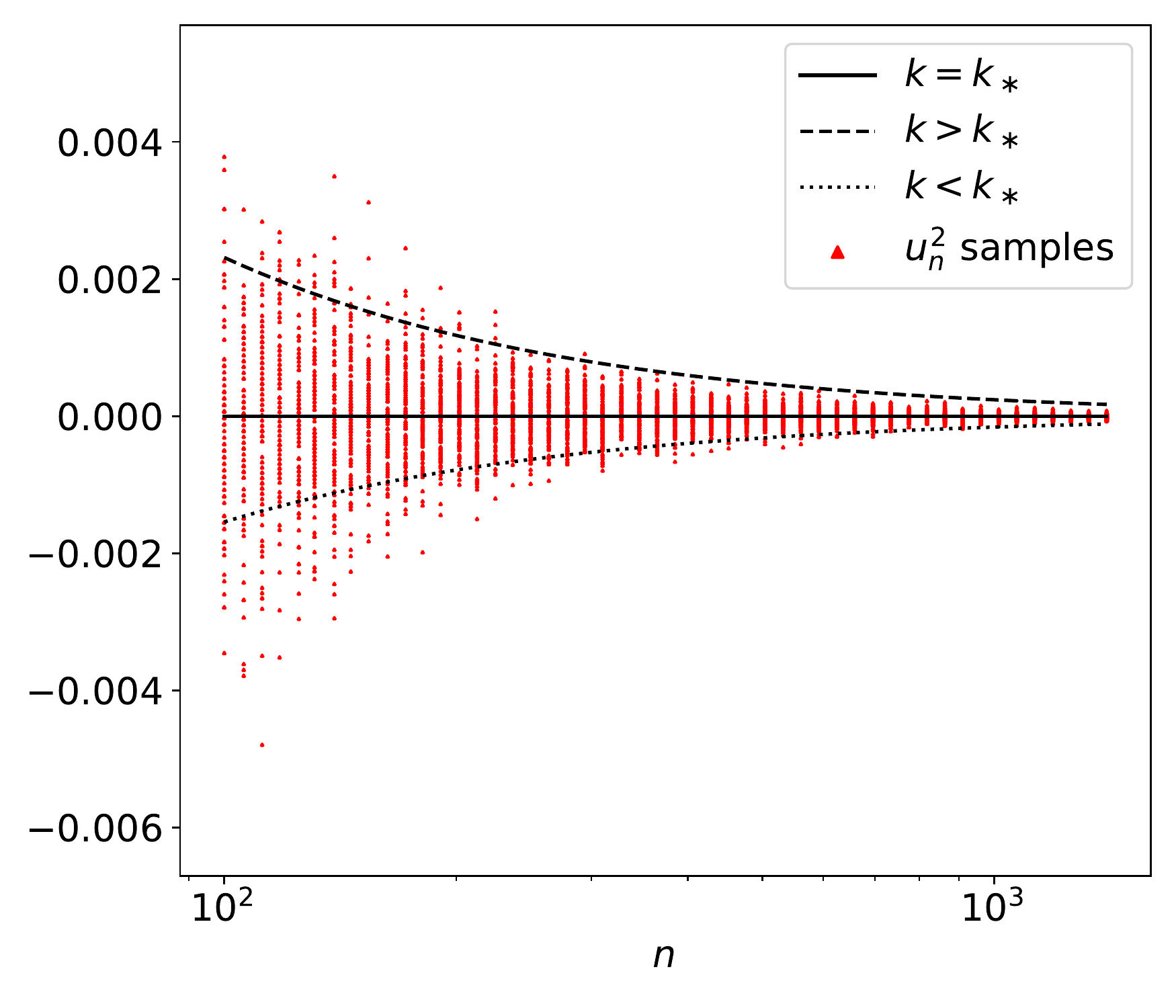}
		\caption{Illustration of Theorem \ref{thm:conservative} for an actual dataset. Samples of $u_n^2$ (in red) for different sample sizes using the football goal dataset and the setup from Section \ref{subsec:Poisson} for a prior $\pi(\theta)$ with mean 2.5 and a standard deviation of 1.0. The black lines mark $\overline{u_n^2 } + k \Delta u_n^2$ for $k=k_\ast = 0.2$ (computed from the full dataset) and $k=k_\ast \pm 0.15$. For illustration purposes everything was shifted by $\overline{u^2_n} + k_\ast \Delta u^2_n$.  }
		\label{fig:theorem_figure}
\end{wrapfigure}
The behavior predicted by Theorem \ref{thm:conservative} might be surprising at first glance: recall that $\overline{u_n^2}$ is defined as an average and that the object  $\Delta u_n^2$ is the according standard deviation. Naively one might expect that the range of $k$ times the standard deviation around the mean might always cover some percentage strictly between $0$ and $1$. The reason behind the observed phenomenon is that we use one distribution, $m(\xx_n)$, to perform the SSD and another one, $p(\xx_n|\theta_{\mathrm{true}})$, to evaluate its success. The different asymptotics between these two distributions then gives rise to the erratic behavior around $k=k_\ast$. The effect for a fixed $k$ and various priors is depicted by Figure \ref{fig:phase_transition} in the appendix: As $\eps$ gets small we get two sharply separated regions for the hyperparameters where one region has a quota of $u_{n_\eps}^2 <\eps^2$ around $0\%$ while the other one approaches $100\%$.
We expect that this phenomenon is not solely adherent to the VPVC. For instance, it seems reasonable that for many cases a credible interval for $u^2_n$ under $m(\xx_n)$ is, at least loosely, related to a certain $k$ so that similar effects are likely to arise for a framework based on other approaches using $m(\xx_n)$. 
Investigating such a point further might be an interesting subject for future research.

We have seen in Theorem \ref{thm:conservative} that for small $\eps$ the success of the SSD crucially depends on having $k>k_\ast$. Now, $k_\ast$ is dependent on the true parameter $\theta_{\mathrm{true}}$ and thus unknown in practice. For the prior used for Figure \ref{fig:theorem_figure} we have a $k_\ast$ around $0.21$ and therefore much smaller than the $k=2$ used in the analysis of Section \ref{subsec:Poisson}. In general one might choose for an area $\Theta$ in the domain of $\pi(\theta)$ as an upper bound for $k_\ast$
\begin{align}
	k_{\ast}^{\tiny \mbox{up. b.}} = \sup_{\theta \in \Theta}  k_\ast(\theta) \,.
\end{align}
Taking a single standard deviation around the prior mean for $\pi(\theta)$ and $\pi(\sigma^2)$ respectively we get for the prior from Figure \ref{fig:theorem_figure} $k_{\ast}^{\tiny \mbox{up. b.}} = 1.2$ and for the one marked by the black cross in Figure \ref{subfig:Music_acc_s2} $k_{\ast}^{\tiny \mbox{up. b.}} =1.0$. For two standard deviations we obtain 2.2 and 2.0 respectively. All these values are well above the actual values of $k_\ast$ which are 0.2 and 0.7 for these two examples and below or at least close to the value of $k=2$ we applied throughout Section \ref{sec:SSD}.

\section{Conclusions} 
We discussed a Bayesian criterion for sample size determination, based on the prior predictive, which we called the variation of posterior variance criterion (VPVC). Compared with the classical average posterior variance criterion (APVC) this criterion leads to a better compliance with the objective to restrain the uncertainty by some $\eps$, while still providing explicit expressions in contrast to a full treatment of the law of the uncertainty under the prior predictive. In particular this allows to treat the asymptotic behavior for small $\eps$ in a generic manner and thus to enhance the understanding of sample size methods based on the prior predictive.

Using two different datasets we discussed the dependency of the sample size determination and its success on the chosen prior and the true parameter and deduced concepts on how to choose the prior to lead the sample size determination more likely to a success. In particular, we observed that a part of these strategies cannot be applied for the APVC but only for methods such as the VPVC, that take higher moments with respect to the prior predictive into account. 

Finally, we gave some results concerning the behavior of the VPVC for $\eps\rightarrow 0$. We proved an explicit formula for the predicted sample size in this regime and showed that there is an exact limit for the portion of the variance w.r.t. prior predictive that can be used in order to guarantee the success of the sample size determination in this limit.

\bibliography{SSD_Post_Var} 
\bibliographystyle{ieeetr}

\newpage
\appendix

\section{Effect of decreasing $\eps$ on $u^2_{n_\eps}<\eps^2$}%
\renewcommand\thefigure{\thesection.\arabic{figure}}    
\setcounter{figure}{0}
\begin{figure}[h!]
\centering
	\begin{subfigure}[t]{0.40\textwidth}
		\includegraphics[width=\textwidth]{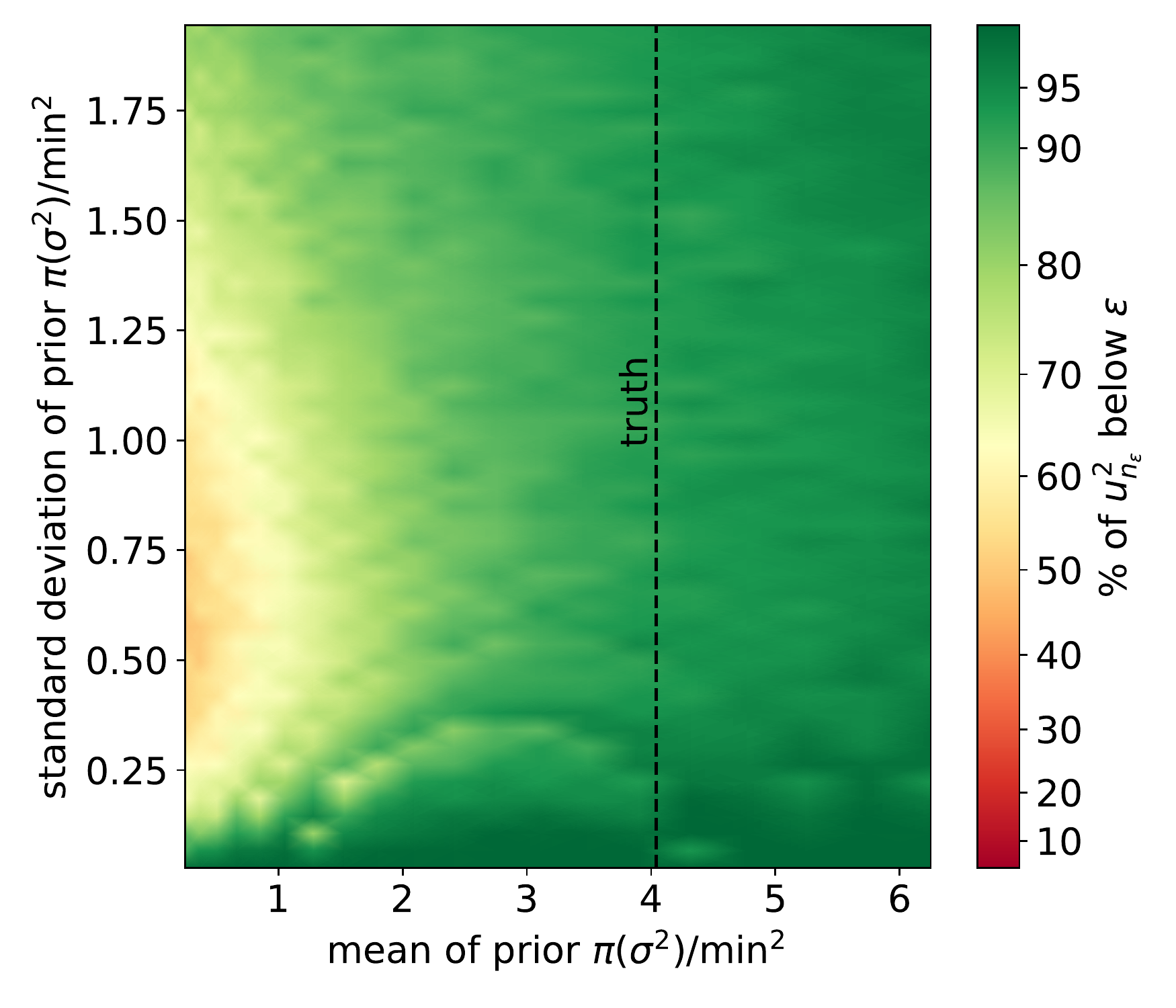}
		\caption{$\eps=30\,\mathrm{sec}$}
	\end{subfigure}%
~
	\begin{subfigure}[t]{0.40\textwidth}
		\includegraphics[width=\textwidth]{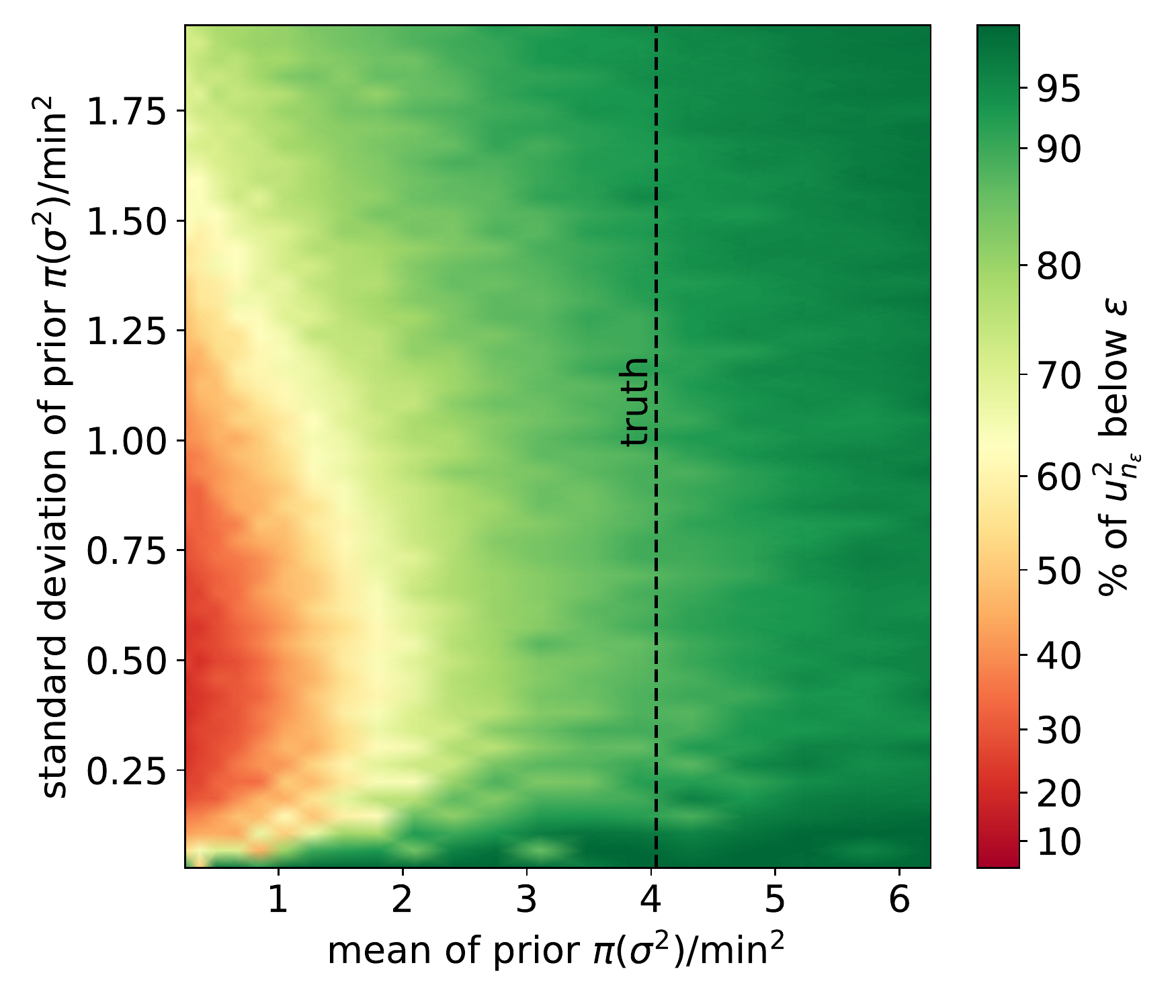}
		\caption{$\eps=20\,\mathrm{sec}$}
	\end{subfigure}
	\begin{subfigure}[t]{0.40\textwidth}
		\includegraphics[width=\textwidth]{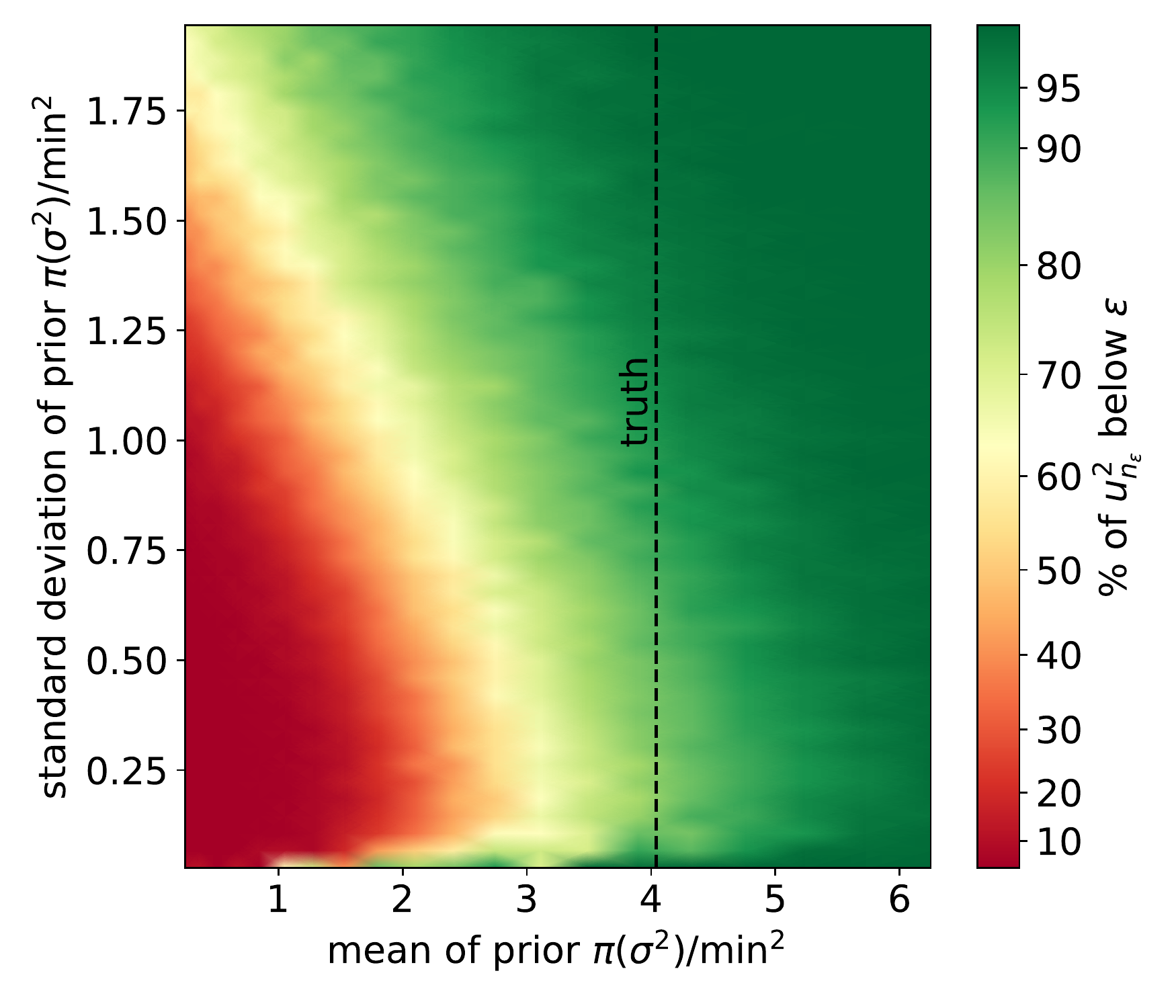}
		\caption{$\eps=10\,\mathrm{sec}$}
	\end{subfigure}%
~
	\begin{subfigure}[t]{0.40\textwidth}
		\includegraphics[width=\textwidth]{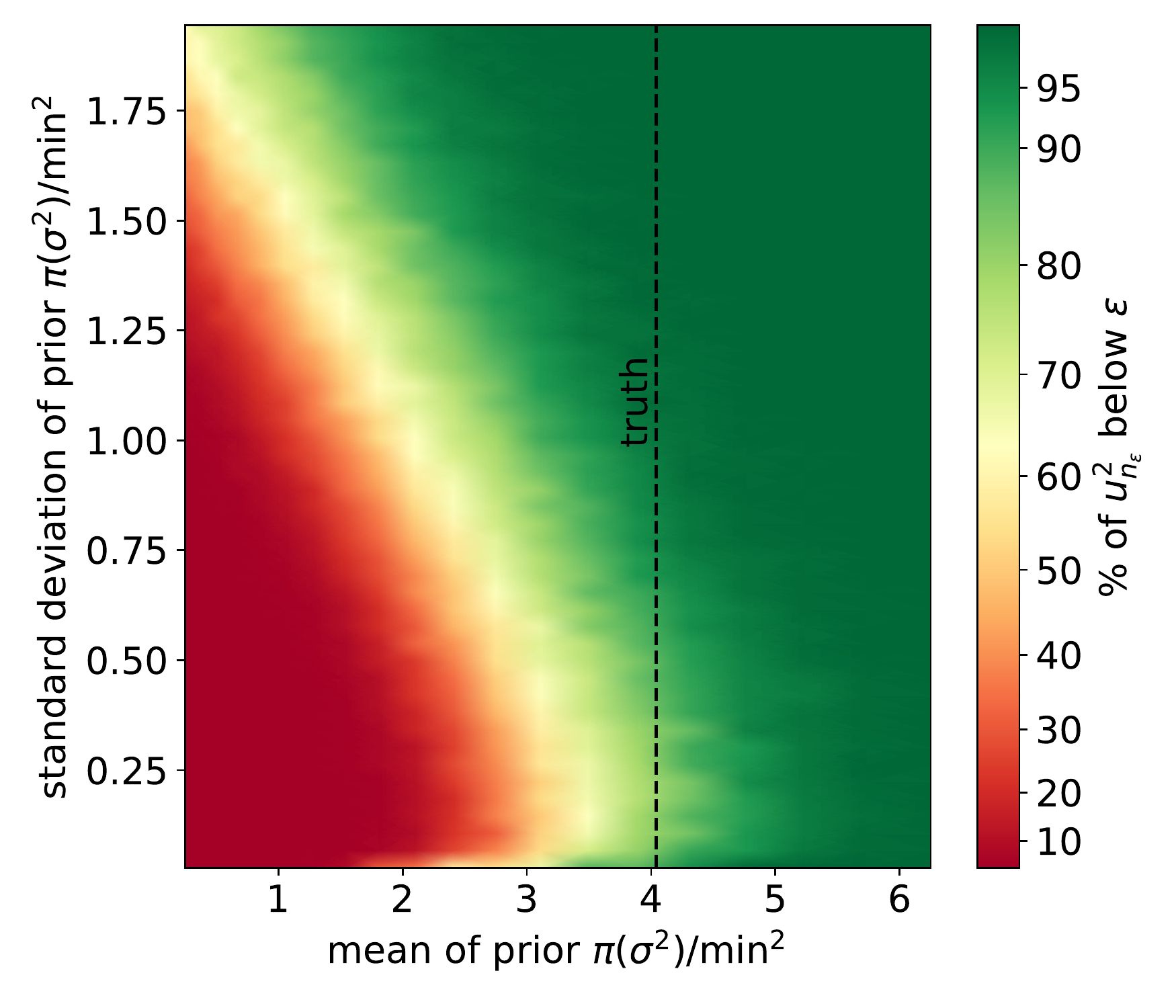}
		\caption{$\eps=7\,\mathrm{sec}$}
	\end{subfigure}
	\begin{subfigure}[t]{0.40\textwidth}
		\includegraphics[width=\textwidth]{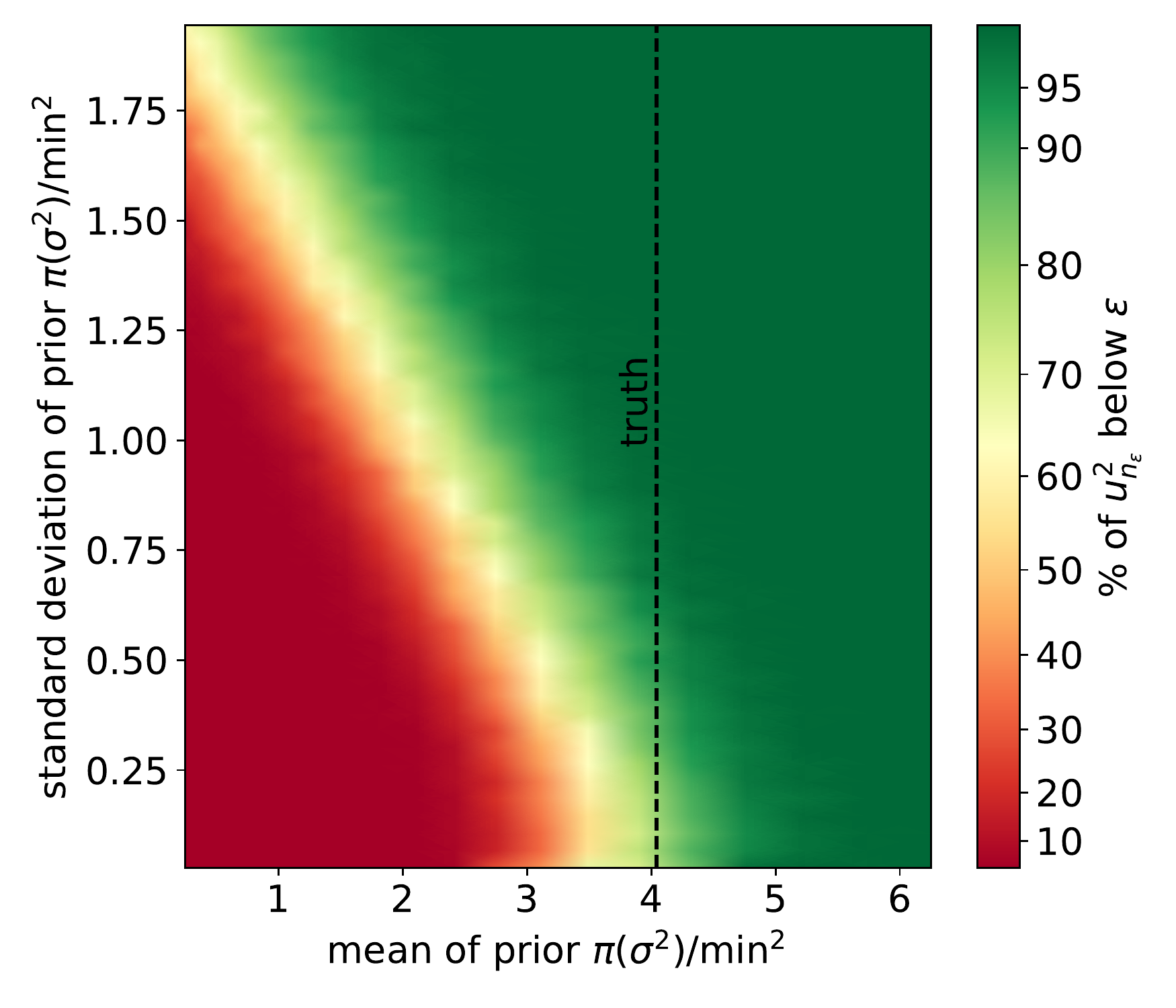}
		\caption{$\eps=5\,\mathrm{sec}$}
	\end{subfigure}%
~
	\begin{subfigure}[t]{0.40\textwidth}
		\includegraphics[width=\textwidth]{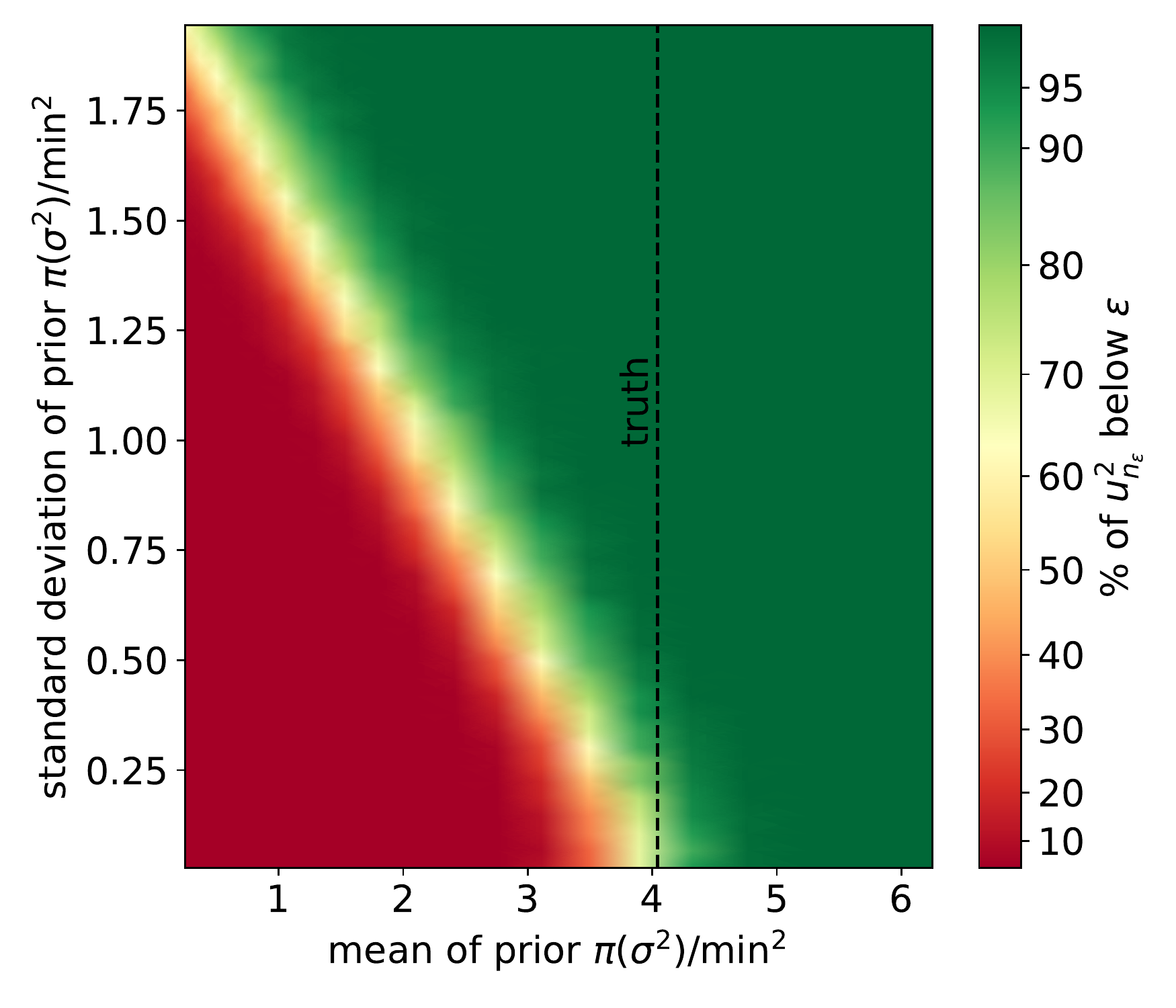}
		\caption{$\eps=3\,\mathrm{sec}$}
	\end{subfigure}
	\label{fig:Music_acc_eps}
	\caption{Percentage of $u_{n_\eps}^2 < \eps^2$ for the setup and data from Section \ref{subsec:normal} for various $\eps$ (in decreasing order), $k=2$,  various marginals $\pi(\sigma^2)$ and the same fixed marginal $\pi(\mu)$ as in Figure \ref{subfig:Music_acc_s2}. The ``true'' value of $\sigma^2$ (computed from the dataset) is depicted by the dashed lines. We used a smaller resolution compared to Figure \ref{subfig:Music_acc_s2} for numerical reasons.}
	\label{fig:phase_transition}
\end{figure}
\end{document}